\documentclass[11pt,reqno]{amsart}
\usepackage[margin=1.4in,letterpaper]{geometry}                
\usepackage{graphicx,comment,cancel}
\usepackage{mathtools} 
\usepackage{amsmath,amssymb}
\usepackage{epstopdf}
\usepackage[colorlinks=true, pdfstartview=FitV, linkcolor=blue, citecolor=blue, urlcolor=blue]{hyperref}
\usepackage{enumitem}
\usepackage{xcolor}
\usepackage{tikz-cd}
\usepackage[all]{xy}
\usepackage{appendix}
\usepackage[nameinlink, capitalize]{cleveref}
    \crefname{conj}{conjecture}{conjectures}
    \crefname{conj}{Conjecture}{Conjectures}
    
\numberwithin{equation}{section}

\newtheorem{thm}{Theorem}[section]
\newtheorem{introthm}{Theorem}

\newtheorem*{introthm*}{Main Theorem}

\newtheorem{cor}[thm]{Corollary}
\newtheorem{lem}[thm]{Lemma}
\newtheorem{prop}[thm]{Proposition}

\theoremstyle{definition}
\newtheorem{defn}[thm]{Definition}
\newtheorem*{defn*}{Definition}

\newtheorem{ex}[thm]{Example}

\newtheorem{constr}[thm]{Construction}

\newtheorem{rem}[thm]{Remark}
\newtheorem{notation}[thm]{Notation}

\theoremstyle{remark}
\newtheorem*{claim1}{Claim 1}
\newtheorem*{claim2}{Claim 2}

\newcommand{\kk}{{\sf{k}}}

\newcommand{\N}{{\mathbb N}}
\renewcommand{\P}{{\mathbb P}}
\newcommand{\Q}{{\mathbb Q}}
\newcommand{\Z}{{\mathbb Z}}

\def\Po{\operatorname{P}}

\newcommand{\I}{{\mathcal I}}

\newcommand{\m}{{\mathfrak m}}

\newcommand{\sym}{\mathfrak{S}}

\renewcommand{\l}{\lambda}

\def\ch{\operatorname{char}}

\def\pd{\operatorname{pd}}
\def\rank{\operatorname{rank}}

\def\reg{\operatorname{reg}}

\def\Ann{{\operatorname{Ann}}}
\def\Tor{{\operatorname{Tor}}}

\def\HF{{\operatorname{\sf{HF}}}}

\def\Hom{\operatorname{Hom}}
\def\Soc{\operatorname{Soc}}

\def\bsa{\operatorname{\boldsymbol{\alpha}}}

\def\lex{{\operatorname{Lex}}}

\def\Span{{\operatorname{Span}}}

\def\l{\lambda}
\def\type{\operatorname{type}}

\DeclareMathOperator{\Gr}{\operatorname{Gr}}

\tikzset{
  symbol/.style={
    draw=none,
    every to/.append style={
      edge node={node [sloped, allow upside down, auto=false]{$#1$}}}
  }
}

\title{General Symmetric Ideals}

\author[A.~Seceleanu]{Alexandra Seceleanu}
\address{
Department of Mathematics, University of Nebraska-Lincoln, Lincoln,  NE 68588, U.S.A.}
\email{aseceleanu@unl.edu}

\author[L.~\c{S}ega]{Liana \c{S}ega}
\address{
Division of Computing, Analytics and Mathematics, University of Missouri-Kansas City, Kansas City, MO 64110, U.S.A.}
\email{segal@umkc.edu}


\date{\today}

\begin{document}
\begin{abstract}
We investigate the structure and properties of symmetric ideals generated by general forms in the polynomial ring 
 under the natural action of the symmetric group. This work significantly broadens the framework established in our earlier collaboration with Harada on principal symmetric ideals \cite{HSS24}. A novel aspect of our approach is the construction of a bijective parametrization of general
symmetric ideals using Macaulay-Matlis  duality, which is asymptotically independent of the number of variables of the ambient ring. We establish that general symmetric ideals exhibit extremal behavior in terms of Hilbert functions and Betti numbers, and satisfy the Weak Lefschetz Property. We also demonstrate explicit asymptotic stability in their algebraic and homological invariants under increasing numbers of variables, showing that such ideals form well-behaved $\sym_\infty$--invariant chains. 
\end{abstract}

\maketitle

\tableofcontents

\section{Introduction}
\label{s:prelim}

Let $\kk$ be a field and let $R=\kk[x_1,\ldots, x_n]$ denote the polynomial ring in $n$ variables. 
Let $I$ be a homogeneous ideal in $R$, i.e., an ideal generated by homogeneous polynomials $f_1, \ldots , f_r$ of degrees $d_1,\ldots, d_r$, respectively. A theme of investigation with a few decades of tradition in commutative algebra concerns the properties of the cyclic module $R/I$ as the polynomials $f_1, \ldots , f_r$ are chosen to be as independent as possible. For example, if $r\leq n$ this means that $f_1, \ldots , f_r$ form a regular sequence and thus $R/I$ is a complete intersection, while for $r\geq n$ the ring $R/I$ is a finite dimensional $\kk$-vector space, thus an artinian $\kk$-algebra.

It was shown by Fr\"oberg and L\"ofwall in \cite{FrobergLofwall} that, for fixed values of  $n, d_1, \ldots , d_r$, there exists only a finite number of possible Hilbert series for $R/I$. Moreover,  there is a Zariski open subset in the space of coefficients of the $f_i$'s on which the Hilbert series of $R/I$ is the same and it is in fact the minimal series among all possible Hilbert series of quotients with the same parameters. Famously, a longstanding conjecture by Fr\"oberg describes this minimal Hilbert series \cite{Froberg} and is related to Lefschetz properties of artinian rings. One calls algebras with this Hilbert series {\em general}. 
In complete generality Fr\"oberg's conjecture has proven elusive, however the following works make partial progress: \cite{Froberg} ($n=2$), \cite{Anick} ($n$=3), \cite{Stanley} ($r=n+1$), \cite{Nenashev} (infinitely many cases with $d_1=d_2=\cdots=d_r$).

There is an additional conjecture due to Migliore and Mir\'o-Roig \cite[Conjecture 5.8]{MM} regarding the betti numbers of general algebras $R/I$ with $I$ equigenerated ($d_1=d_2=\cdots=d_r$). It asserts that such algebras would possess minimal Betti numbers consistent with the Fr\"oberg Hilbert function. A  number of particular cases are elucidated in \cite{MM}.

Another method to specify an ideal $I$ defining an artinian quotient $R/I$ is by means of its Macaulay inverse system $I^{-1}$ (see \eqref{eq:Iperp}). Macaulay-Matlis duality relates the $R$-modules $R/I$ and $I^{-1}$ to each other. Iarrobino defines compressed algebras as a class of graded algebras parametrized by their Macaulay inverse systems that exhibit extremal behavior in their Hilbert function. Specifically, compressed algebras have the maximal Hilbert function possible given the dimension of their socle and the embedding dimension. These rings are studied in \cite{EI, Iarrobino, FrobergLaksov, Green}, among others. With respect to the parametrization corresponding to the inverse system compressed algebras are general, meaning that they form a dense Zariski open set.
  In \cite{Boij} Boij studies the betti numbers of compressed level algebras giving lower bounds for these numbers and finding that  these lower bounds are attained by compressed algebras of large enough socle degree. This  leaves open several questions which are still being investigated.  In \cite{RS} Rossi and \c{S}ega study resolutions of modules over compressed Gorenstein local rings showing that all such modules have rational Poincar\'e series with a common denominator.

Inspired by these well-known problems and results, our goal in this paper is to discuss analogous phenomena concerning graded algebras equipped with an action of the symmetric group. To construct such algebras, consider 
the natural action of the symmetric group $\sym_n$ on the polynomial ring $R=\kk[x_1,\ldots, x_n]$ given by 
$$
\sigma \cdot f(x_1,\ldots, x_n)=f(x_{\sigma(1)}, \ldots, x_{\sigma(n)})\quad \text{for}\quad  \sigma\in \sym_n\,. 
$$
 We are interested in homogeneous ideals which acquire an induced action from the action of $\sym_n$ on $R$, which we refer to as symmetric ideals. The study of such ideals is both a classical topic in commutative algebra and one of current interest. 
 In particular, the graded betti numbers for (certain families of) monomial symmetric ideals are the focus of the recent works  \cite{Galetto}, \cite{Satoshi1}, \cite{Biermann} and \cite{Satoshi-Raicu} and related invariants such as Castelnuovo-Mumford regularity are studied in \cite{Raicu1, Nagel2}. 
The symmetric ideals discussed in this paper behave quite differently from most symmetric monomial ideals and therefore novel methods are introduced to study them in our work.  We were inspired to study general symmetric ideals by \cite{Kretschmer}, which describes the radical of such an ideal.  
 
 In this paper, we focus initially on homogeneous symmetric ideals having all minimal generators of the same degree, deducing later that this is not a restrictive assumption (see \cref{rem: multiple}).  We term {\em $(r,d)$-symmetric ideals} those symmetric ideals that can be generated up to the action of the symmetric group by an $r$-dimensional vector space of homogeneous polynomials of degree $d$. Such ideals can be parametrized by the Grassmannian variety of $r$-dimensional subspaces of the $N=\binom{n+d-1}{d}$-dimensional vector space of forms of degree $d$ forms, $R_d$. This parametrization is given by a map $\Phi$ that takes a vector space $V$ to the ideal  $(V)_{\sym_n}$ generated by the orbits of elements of $V$ under the action of the symmetric group $\sym_n$. To summarize, we have an onto map:
 \begin{align}\label{Grassmannian intro}
&\Phi: \Gr(r,R_d) \to (r,d) \text{-symmetric ideals} \\
& \Phi(V) :=(V)_{\sym_n}\,,\qquad \text{where }  (V)_{\sym_n}=(\sigma\cdot f : \sigma\in \sym_n, f\in V). \nonumber
\end{align}
The case $r=1$ has been considered in our previous work \cite{HSS24} joint with M.\,Harada, where the $(1,d)$-symmetric ideals are called {\em principal symmetric ideals}. 

Another way to parametrize symmetric ideals stems from Macaulay duality. The inverse system of a symmetric ideal of $R$ is an $R$-submodule of a dual  ring $S$ (see \ref{eq:contraction}). This ring is also equipped with an action of $\sym_n$. In \cite{HSS24} we discovered that the inverse system of a principal symmetric ideal is generated by $\sym_n$-invariant polynomials in $S$. Therefore we consider the rings of invariant polynomials $R'=R^{\sym_n}$ and $S'=S^{\sym_n}$. Their elements form the image of the {\em Reynolds operator} $\rho(f)=\frac{1}{n!} \sum_{\sigma\in \sym_n} \sigma \cdot f$ applied to $R$ and $S$ respectively. A key component of our approach is to parametrize the symmetric ideals of $R$ that have inverse systems generated by elements of $S'$. Specifically we term a quotient $R/I$ a {\em narrow algebra with $(s,d)$--invariant socle} if the inverse system has the form $I^{-1}=W+S_{\geq-d+1}$, where $W$ is a vector subspace of $S'_{-d}$ and $\dim_\kk(W)=s$.

Working with $R'$ and $S'$ presents the advantage that up to isomorphism the graded vector spaces $R'_d$ and $S'_{-d}$ only depend  on $d$ and not on $n$, provided that $n\geq d$. This  allows to parametrize narrow algebra having $(s,d)$--invariant socle by a significantly smaller space than the source of \eqref{Grassmannian intro} above, 
which has the advantage of being {\em uniform}, that is, independent  of the number of variables and dependent only upon $s$ and $d$. To set this up, consider the following map induced by $\rho$ 
 \begin{equation}\label{alpha map}
\alpha: \Gr(r,R_d)\to \bigcup_{0\leq i\leq r} \Gr(i,R'_d)
\qquad \alpha(V):=\langle \rho(f) : f\in V\rangle.
\end{equation}
Moreover we denote by $\beta$ the  map 
  \begin{align}\label{beta map}
&\beta:  \Gr(P(d)-r,S'_{-d}) \to  \text{symmetric ideals}   \\
&\beta(W):= \Ann_R(W+S_{\geq -d+1}), \nonumber
\end{align}
where the operator $\Ann_R(-)$ implements Macaulay-Matlis duality  as described in \ref{s: duality} and $P(d)$ denotes the number of partitions of $d$ with at most $n$ parts, which is equal to $\dim_\kk R'_d$ and also $\dim_\kk S'_{-d}$. 

Assuming that $n$ is sufficiently large, in the course our investigation we find nonempty Zariski open sets $O\subset \Gr(r,R_d)$ and $G\subset \Gr(r,R'_d)$ so that $\alpha(O)= G$,   
and diagram \eqref{eq: CD} below commutes.

We call the ideal $\Phi(V)=(V)_{\sym_n}$ a {\em general $(r,d)$-symmetric ideal} whenever $V\in O$, see \Cref{d: def-gsi}. 

\begin{equation}\label{eq: CD}
\begin{tikzcd}[row sep =9mm, column sep = 2.3mm]
  \Gr(r,R_d)\arrow[r,symbol=\supseteq]   & O \arrow[]{rrrrrrrrrr}{\Phi}\arrow[twoheadrightarrow, d, "\alpha"'] &&&&&&&&&& (r,d)\!\!\!\!\!\arrow[r,symbol=\text{-}]&\text{\!\!\!\!\!symmetric ideals}  \\
    \Gr(r,R'_d)\arrow[r,symbol=\supseteq] & G  \arrow{rrrrrrrrrr}{\perp} \arrow{urrrrrrrrrr}{\Psi} 
    &&&&&&&&&& G^\perp\! \arrow[hookrightarrow, u, "\beta"']\arrow[r,symbol=\subseteq]  & \!\Gr(P(d)-r,S'_{-d}).
\end{tikzcd}
\end{equation}

Denoting the composition of $\perp$ and $\beta$ in \eqref{eq: CD} by $\Psi$ gives an alternate parametrization
\begin{equation}\label{first map psi}
    \Psi:G \to (r,d)\text{-symmetric ideals}.
\end{equation}
In view of the surjectivity of $\Psi$, which we establish in \Cref{t: CD}, every general symmetric ideal is of the form $\Psi(U)$ for a unique $U\in G$. This shows that, like $\Phi$, $\Psi$ is a parametrization of the general $(r,d)$-symmetric ideals, but it is superior to $\Phi$ because $\Psi$ is one-to-one and the parametrizing space $G$ is independent of the number of variables, provided $n\geq d$.

Our results regarding diagram \eqref{eq: CD} can be summarized as follows.

\begin{introthm}[\Cref{t: CD}]\label{thmA}
Suppose $\kk$ is  infinite  with ${\rm char}(\kk)=0$ or $\ch(\kk)>n$ and let $d,r$ be positive integers with $r\le P(d)$. Assume $n$ is sufficiently large. Then there exist nonempty open sets $O\subseteq \Gr(r,R_d)$ and $G\subseteq \Gr(r,R'_d)$ such that 
\begin{enumerate}
    \item diagram  \eqref{eq: CD} commutes, 
    \item $\Psi(G)=\Phi(O)$,
    \item $\alpha(O)=G$ and
    \item  $\Psi|_G$ is bijective and thus  \eqref{first map psi} gives a birational parametrization of $(r,d)$-symmetric ideals in terms of the parameter space $\Gr(r, \kk^{P(d)})$. 
\end{enumerate}

\end{introthm}

The parametrization through the map $\Phi$ is well suited to establishing algebraic and homological properties of general $(r,d)$-symmetric ideals. 
We summarize them in \Cref{thmB} below, noting that item (2) says that general $(r,d)$-symmetric ideals are extremal in terms of their Hilbert functions among all symmetric ideals with the same number of generators up to symmetry, while item (3) presents the betti numbers of such ideals as being nearly  minimal  consistent with the Hilbert function in (2) (the minimal such betti numbers would additionally satisfy $\ell\cdot b=0$).

\begin{introthm}[\Cref{thm: main}, \Cref{cor: extremal HF}]\label{thmB}
Suppose $\kk$ is  infinite  with ${\rm char}(\kk)=0$ or $\ch(\kk)>n$ and fix positive integers $d$, $r$. Assume 
\[
n\ge  \max\left\{d+1, 1+\frac{1}{r} \sum_{i=0}^{d-1} P(i)\right\}\,.
\]
Then a general $(r,d)$-symmetric ideal $I$ of ${\kk}[x_1,\ldots, x_n]$ satisfies the following: 

\begin{enumerate}
\item The Hilbert function of $A=R/I$ is given by
\[
 \HF_A(i):=\dim_\kk A_i = \begin{cases}\dim_\kk R_i &\text{if $i\le d-1$}\\
\max\{ P(d)-r,0\}&\text{if $i=d$}\\
 0 &\text{if $i>d$.}\
 \end{cases}
\]
\item $A$ has the smallest Hilbert function among all cyclic $R$-modules  $R/J$ with $J$ an $(r,d)$-symmetric ideal.
\item The betti table of $A$ has the form
\[
\begin{matrix}
     &0&1&2&\cdots&i &\cdots & n-2 &n-1&n\\
     \hline
     \text{total:}&1&u_1&u_2&\cdots&u_i&\cdots&u_{n-2} &u_{n-1}+\ell&a+b\\
     \hline
     \text{0:}&1&\text{.}&\text{.}&\text{.}&\text{.}&\text{.}&\text{.}&\text{.}&\text{.}\\
     \text{\vdots}&\text{.}&\text{.}&\text{.}&\text{.}&\text{.}&\text{.}&\text{.}&\text{.} &\text{.}\\
     \text{d-1:}&\text{.}&u_1&u_2&\cdots&u_i&\cdots&u_{n-2} &u_{n-1}&b\\
     \text{d:}&\text{.}&\text{.}&\text{.}&\text{.}&\text{.}&\text{.}&\text{.}&\ell & a \\
     \end{matrix}
\]
with 
\begin{eqnarray*}
\ell &=& \max\{P(d)-P(d-1)-r,0\} \\
a &=&\max\{P(d)-r,0\},\\
 b&=&
 \binom{n+d-2}{d-1}-an+ \ell  \\
  u_{i+1} &=& \binom{n+d-1}{d+i}\binom{d+i-1}{i}-a\binom{n}{i}. 
  \end{eqnarray*}
\item $A$ has the Weak Lefschetz Property. 
\end{enumerate}

Moreover,  the Poincar\'e series of all finitely generated graded $A$-modules are rational, sharing a common denominator. When $d>2$ or $d=1$ the ring $A$ is Golod. When $d=2$, then $A$ is Koszul. The ring $A$ is Gorenstein if and only if $d=2$ and $r=1$, or $d=1$.
\end{introthm}

Our previous paper \cite{HSS24} contains similar results in the case when $r=1$, that can be recovered from \cref{thmB} and served as inspiration for our approach. The recent paper \cite{Walker} by N. Walker leads to the  effective bound on $n$ given in \cref{thmB} and shows this bound is tight, while independent results presented in this paper establish the validity of \cref{thmB} for large enough $n$ without providing an explicit value.

To generalize our previous results, we now let the number of variables of the polynomial ring $R$ grow to infinity. To be more precise, for each positive integer $n$ we let $R_n=\kk[x_1, \ldots, x_n]$.  The colimit of the rings $R_n$ under their natural inclusions, where $n$ ranges over the positive integers, admits a natural action of the group $\sym_{\infty}=\bigcup_{n\geq 1}\sym_n$ of permutations of $\N$ that have finitely many non-fixed points. This action is of utmost importance in the study  of representation stability for symmetric group representations. A major achievement in this setting is the development of Noetherianity up to symmetry for ideals of this (non-Noetherian) colimit and its applications to uniformity in commutative algebra; see \cite{Cohen}, \cite{Hillar}, \cite{HS}, \cite{Erman}, \cite{Draisma}. 

If $I$ is a symmetric ideal of $R_m$ for some $m$, we define a sequence
\begin{equation}\label{eq: sym-inf chain}
    I=I_m\subseteq I_{m+1}\subseteq \dots I_n \subseteq \dots 
\end{equation}
where $I_n$ is the ideal of $R_n$ for all $n\ge m$ described by $I_n=(I_m)_{\sym_n}$.
The chain \eqref{eq: sym-inf chain} is a $\sym_{\infty}$-invariant chain, meaning that $\sym_{n'}(I_n)\subseteq I_{n'}$ for all $m\leq n\leq n'$. Properties of $\sym_{\infty}$-invariant chains have been studied from the viewpoint of commutative algebra in \cite{NagelRomer, Nagel1, Nagel2, Raicu1, Satoshi-Raicu}, among many others. A main problem in this area is to study the asymptotic growth of various numerical invariants for $\sym_{\infty}$-invariant chains of homogeneous ideals.

In view of our results, a natural question raised by C. Raicu is whether $I$ being a general $(r,d)$-symmetric ideal implies the same is true about the subsequent ideals in the chain \eqref{eq: sym-inf chain}. While this is far from evident from the perspective of the parametrization \eqref{Grassmannian intro},  we succeed in answering this question in the affirmative by making use of the alternate parametrization \eqref{first map psi}, particularly its feature  that the parameter space $\Gr(r,R'_d)$ is independent of the number of variables of $R$, provided $n\geq d$. This leads to the realization that diagrams of the form \eqref{eq: CD} where the number of variables of $R$ varies can be made compatible in a manner made precise in \Cref{thm: CD-n}. This  crucial ingredient allows us to analyze the asymptotic behavior of numerical invariants in chains of general $\sym_{\infty}$-invariant ideals, answering \cite[Problem 1.1]{KLR} for chains of general symmetric ideals.

\begin{introthm}[\Cref{thm: CD-n}, \Cref{cor: sym-infty chain}]\label{thmC}
Suppose $\kk$ is a field  with ${\rm char}(\kk)=0$ and let $d,r$ be positive integers with $r\le P(d)$.
Assume $m\gg 0$.  Let $I$ be a general $(r,d)$-symmetric ideal of $R_m$, and consider the $\sym_{\infty}$-invariant chain $\{I_n\}_{n\in \N}$ with $I_n=(I)_{\sym_n}$  and its colimit $\mathcal I$.  For $a, \ell$ as defined in \Cref{thmB}(3) the following  hold: 
\begin{enumerate}
\item the ideals $I_n=(I_m)_{\sym_n}$ are general $(r,d)$-symmetric ideals for all $n\geq m$.
 \item $\dim(R_n/I_n)=0$  and ${\rm ht}(I_n)=n$ for all $n\ge m$, i.e, the rings $R_n/I_n$ are artinian. 
 \item $\pd_{R_n}(R_n/I_n)=n$ for all $n\ge m$
 \item $\reg_{R_n}(R/I_n)=d$ for all $n\ge m$
 \item the $i$-th betti numbers of $R_n/I_n$ are polynomials in $n$ of degree $d+i$ for $0\leq i\leq n-1$, while the $n$-th betti number of $R_n/I_n$ is a linear function of $n$.
 \item The equivariant  Hilbert series (defined in \eqref{eq: equiv HS}) is 
\[
H_\I(s,t)=\frac{s\left[\sum_{j=0}^d(1-s)^jt^{d-j}\right ]+a(1-s)^{d-1}t^d}{(1-s)^d}.
\]
 \item  The  multiplicity of $R_n/I_n$ grows polynomially as a function of $n$ with
 \[
 \lim_{n\to \infty}\frac{e(R_n/I_n)}{n^{d-1}}=\frac{1}{(d-1)!}.
 \]
 \item The equivariant betti number $\beta_{i,j}(\I,s)$ (defined in \eqref{equiv-beta}) is a rational function of $s$ for all $i,j$, and the equivariant bigraded Poincar\'e series $\Po_\I(s,t,u)$  is a rational function of $s,t,u$.  More precisely,
\begin{equation*}
P_\I(s,t,u)\!=\!\frac{1}{1-s}\!+\!stu^d\left(\frac{1}{(1\!-\!s)(1\!-\!s\!-\!stu)^d}\!-\!\frac{a}{(1\!-\!s\!-\!stu)(1\!-\!stu)}\!+\!\frac{\ell+ au +\ell su}{1-stu}\right).
\end{equation*}
\end{enumerate}
\end{introthm}

\Cref{thmC}  shows that $\sym_\infty$-invariant chains of general symmetric ideals are much better behaved than arbitrary $\sym_\infty$-invariant chains. For instance, for our chains the Krull dimension and regularity are constant, while for arbitrary chains the former exhibits linear growth \cite[Theorem 7.10]{NagelRomer} and the latter is proven in the artinian case and conjectured in general to grow linearly \cite{Nagel2}. Moreover for our chains the multiplicity grows polynomially as a function of $n$, while for arbitrary $\sym_\infty$-invariant chains the growth is exponential \cite[Theorem 7.10]{NagelRomer}.

Our paper is organized as follows. In \cref{sec: inverse} we recall the technical foundations of our work: Macaulay-Matlis duality and the notions of narrow and extremely narrow algebras introduced in \cite{HSS24}. We also prove a new result that the latter class of algebras enjoys the Weak Lefschetz Property, an algebraic counterpart of Lefschetz's hyperplane section theorem. In \cref{sec:orthog} we discuss a perfect pairing among symmetric polynomials and in \cref{sec:linear} we use it to study linear syzygies on symmetric polynomials. In \cref{sec:narrow-inv} we define the open set $G$ in \eqref{eq: CD}, show it does not depend on $n$ as long as $n\geq d$, and prove that algebras in $\Psi(G)$ are $d$-extremely narrow (\cref{thm: narrow}). Under the additional assumption that $n$ is sufficiently large so as to enable a certain construction, we determine in \cref{s: equations} the generators of a general symmetric ideal. 
Turning our attention to the parametrization $\Phi$ in \eqref{Grassmannian intro}, we show in \cref{s: diagram} that the open set $O$ in \eqref{eq: CD} is non-empty and prove \Cref{thmA} and \Cref{thmB}. \cref{s: inverse system} reveals the surprising fact that if $n\gg0$ and $I$ is a general symmetric ideal then the inverse system of the ideal $I'=\rho(I)$ of $R'$ determines the inverse system of $I$. This can be viewed as a manifestation of the fact that the parametrization $\Phi$ factors through  $\alpha$ as indicated in \eqref{eq: CD} and  explains why inverse systems of general symmetric ideals are generated by invariant polynomials.   Finally, in \cref{s: asymptotic} we turn our attention to $\sym_\infty$ invariant chains of symmetric ideals and prove \Cref{thmC}.

\section{Inverse systems and narrow algebras}
\label{sec: inverse}
Throughout, $\kk$ denotes a field, $n\ge 2$ an integer and $R=\kk[x_1, \dots, x_n]$ the polynomial ring in $n$ variables of degree $1$ over $\kk$. We denote by $\m$ the maximal homogeneous ideal of $R$. In this section, we recall the basics of the theory of Macaulay inverse systems. Then we give an overview of narrow and extremely narrow algebras as introduced in \cite{HSS24}, extend some of the results  therein and establish that extremely narrow algebras satisfy the Weak Lefschetz Property.

\subsection{Duality and inverse systems}\label{s: duality}
We give here an overview of Macaulay inverse systems. For details, consult \cite[Appendix A]{IK} and \cite[Theorem 21.6]{Eisenbud}. Let $S=\kk[y_1, \dots, y_n]_{DP}$ denote the divided powers algebra in $n$ divided powers variables of degree $-1$ over $\kk$. We equip $S$ with an  $R$-module structure, called {\it contraction}, defined by extending the following action on monomials, by linearity in both arguments: 
\begin{equation*}
\label{eq:contraction}
x_1^{d_1}\cdots x_n^{d_n}\circ y_1^{(e_1)}\cdots y_n^{(e_n)}=
\begin{cases}
y_1^{(e_1-d_1)}\cdots y_n^{(e_n-d_n)} & \text{if } e_i-d_i\geq 0 \text{ for all } i\\
0 & \text{otherwise\,. }
\end{cases}
\end{equation*}
It is well-known that $S=\Hom_\kk(R,\kk)$ is an injective hull of $\kk$ over $R$, when considered as an $R$-module. We refer to the basis elements $y_1^{(e_1)}\cdots y_n^{(e_n)}$ of $S$ as monomials, noting that they are dual to the monomials $x_1^{e_1}\cdots x_n^{e_n}$, and we refer to the elements of $S$ as polynomials $g(y_1, \dots, y_n)$ in the (divided powers) variables $y_i$.   

If $I$ is a homogeneous ideal of $R$, then the {\it Macaulay inverse system} of $I$ is the graded $R$-submodule of $S$ given by 
\begin{equation*}
\label{eq:Iperp}
I^{-1}:=\Ann_S(I)=\{ g\in S : f\circ g=0 \text{ for all } f\in I\}. 
\end{equation*}
Since $\Ann_S(I)\cong \Hom_R(R/I, S)$, the inverse system $I^{-1}$ is isomorphic to the Matlis dual of $R/I$. It is thus a consequence of Matlis duality that 
\begin{equation*}
\label{e:duality}
 \Ann_R(I^{-1})=I \qquad \text{and}\qquad (\Ann_R(W))^{-1}=W\,.
 \end{equation*}
where $W$ is any graded $R$-submodule of $S$. 

Let $U$ be a finitely generated graded $R$-module. The {\it Hilbert function} of $U$ is the  function 
$$\HF_U\colon \mathbb Z\to \mathbb N\,,\qquad \text{with}\qquad \HF_U(i)=\dim_\kk U_i\quad\text {for all $i\in \mathbb Z$}\,.
$$
If $A=R/I$ is artinian, the Hilbert function of $A$ and of its socle $\Soc(A)=\Ann_A(\m)$ can be read off the inverse system as follows: 
\begin{gather}
\label{eq:GIperp}
\begin{split}
\HF_{I^{-1}}(-i)=\dim_\kk (I^{-1})_{-i}=\dim_\kk A_i =\HF_A(i)\\
\dim_\kk\left(\frac{I^{-1}}{\m\circ I^{-1}}\right)_{-i}=\dim_\kk \Soc(A)_i=\HF_{\Soc(A)}(i)
\end{split}
\end{gather}
The socle polynomial of an artinian algebra is the Hilbert series of the socle, namely it is the polynomial
\[
\sum_{i=0}^{s(A)}\HF_{\Soc(A)}(i)z^i \,.
\]
where $s(A)$ is the largest $i$ with $\Soc(A)_i\ne 0$, called the {\it top socle degree} of $A$. 

\subsection{Narrow and extremely narrow algebras}
\label{def:narrow}
Let $A=R/I$ be an artinian algebra quotient of $R$, with $I$ a homogeneous ideal. Denote by $s(A)$ the top socle degree of $A$ and by $t(I)$ be the initial degree of $I$, which is defined as the smallest degree of a minimal generator of $I$. Following \cite{HSS24}, we say that $A$ is {\it narrow} if $t(I)\ge s(A)$, that is, the top socle degree of $A$ is at most the initial degree of $I$.

Set $d=s(A)$. Let $F_1, \dots, F_a\in (I^{-1})_{-d}$   and define 
\[
L_{F_1, \ldots, F_a}:=\left \{(\ell_1, \ell_2, \dots, \ell_a)\in {R_1}^a\colon \sum_{i=1}^a \ell_i\circ F_i=0 \right \} \,.
\]
The $\kk$-vector space $L_{F_1, \ldots, F_a}$  is independent, up to isomorphism, of the choice of the forms $F_1, \dots, F_a$. When working with properties invariant under isomorphism (for example, when talking about the dimension of this space), we will also write $L_W$ instead of $L_{F_1, \ldots, F_a}$, where $W$ is the vector space spanned by $F_1, \dots, F_a$. Furthermore, when $W=(I^{-1})_{-d}$, we also write $L_A$ instead of $L_W$.

Let  $d\ge 1$. As in \cite{HSS24}, we say that $A$ is $d$-{\it extremely narrow} if the following hold:
\begin{enumerate}
\item[(i)] $s(A)= t(I) =d$; 
\item[(ii)] There exists a basis $F_1, \dots, F_a$ of  $(I^{-1})_{-d}$ and $x\in R_1$ such that
\begin{equation}
\label{eq:Lx}
 L_{F_1, \ldots, F_a}\subseteq x(R_0)^a.
\end{equation}
\end{enumerate} 

The result below is a combination of part(s) of the statements of \cite[Lemma 4.2, Lemma 4.5, Lemma 4.7]{HSS24}. 
 
\begin{prop}[see \cite{HSS24}]
\label{prop: narrow}
Let $d\ge 1$, $a\ge 1$, and $F_1, \dots, F_a\in S_{-d}$ linearly independent. If 
\[
I=\Ann_R\left((F_1 \dots, F_a)+S_{\geqslant-(d-1)}\right),
\]
then the algebra $A=R/I$ is narrow with $s(A)=d$ and $\dim_\kk(I^{-1})_{-d}=a$,  and the following hold: 
\begin{enumerate}
 \item The socle polynomial of $A$ is $(\dim_\kk R_{d-1}-an+\dim_\kk L_{A})z^{d-1}+az^{d}$. 
 \item $\dim_\kk I_d=\dim_\kk R_d-a$
 \end{enumerate}
 Furthermore, if there exists $x\in R_1$ such that $L_{F_1, \dots, F_a}\subseteq x(R_0)^a$, then $A$ is $d$-extremely narrow and, when $n\ge 3$, $I$ is generated in degree $d$.  
 \end{prop} 

 \begin{rem}
 \label{rem:L=0}
Assume the hypotheses of \cref{prop: narrow}. If $L_{F_1, \dots, F_a}=0$, then the conclusion that $I$ is generated in degree $d$, and hence $A$ is $d$-extremely narrow, also holds when $n<3$. Indeed, since $A$ is narrow, it follows from \cite[Lemma 4.2]{HSS24} that $I$ is generated in degrees $d$ and $d+1$. Further, since $L_{F_1, \dots, F_a}$ is the module of linear syzygies on a basis of $(I^{-1})_{-d}$ and $I^{-1}$ is generated in degrees $-d$ and $-d+1$,  we see that $\Tor_1^R(I^{-1}, \kk)_{-d+1}=0$, and hence  $\Tor_i^R(I^{-1},\kk)_{-d+i}=0$ for all $i\ge 1$. Then, as noted in the proof of \cite[Lemma 3.13]{HSS24} it is a consequence of Matlis duality that  
\[
\Tor_i^R(A,\kk)_j\cong \Tor_{n-i}^R(I^{-1}, \kk)_{n-j}
\]
for all $i,j$, and hence 
\[
\Tor_0^R(I,\kk)_{d+1}\cong \Tor_1^R(A,\kk)_{d+1}\cong \Tor_{n-1}^R(I^{-1},\kk)_{n-1-d}=0\,.
\]
Thus $I$ is generated in degree $d$. 
 \end{rem}

\subsection{Weak Lefschetz Property}
 The algebra $A$ is said to have the {\it Weak Lefschetz Property} (abbreviated WLP) if the multiplication $A_{i-1}\xrightarrow{\ell} A_i$ by a general linear form $\ell$ has maximal rank for all $i$. More precisely, if we parametrize the elements of $R_1$ by a projective space $\mathbb P^{n-1}$ using their coefficients, then $A$ has WLP if there exists a nonempty open subset $U$ of $\mathbb P^{n-1}$ such that $A_{i-1}\xrightarrow{\ell} A_i$ for all $i$ and all $\ell$ parametrized by $U$. In general, neither homological invariants of artinian rings such as their betti numbers nor information on their Hilbert functions can be used to determine whether an algebra satisfies the WLP; see \cite[Example 3.4]{Wiebe} or \cite{AbdallahSchenck}. However, in the case of $d$-extremely narrow algebras we are able to prove the WLP, and we can also describe the non-Weak Lefshetz locus. 
\begin{prop}
\label{p: e-narrow-has-WLP}
If $A=R/I$ is $d$-extremely narrow, then $A$ has the Weak Lefschetz Property. More precisely, if $\ell\in R_1$ is not a multiple of $x$, where $x\in R_1$ is as in \eqref{eq:Lx}, then the map $A_{i-1}\xrightarrow{\ell} A_i$ has maximal rank for all $i$. 
\end{prop}
\begin{proof}
By \cref{prop: narrow} and the definition of a $d$-extremely narrow algebra, we know that $I$ is generated in degree $d$, and $A$ has socle in degree $d$. Thus, 
$A_i=R_i$ for $i<d$ and $A_i=0$ for $i>d$. In particular, the map  $A_{i-1}\xrightarrow{\ell} A_i$ given by multiplication by any linear form $\ell$ is injective for $i<d$. 
To prove the WLP, it suffices to show the map  $A_{d-1}\xrightarrow{\ell} A_d$ given by multiplication by a general linear form $\ell$ is surjective. Since $I^{-1}$ is the Matlis dual of $R/I$, this is equivalent to showing that the map $(I^{-1})_{-d}\xrightarrow{\ell} (I^{-1})_{-d+1}$ given by $F\mapsto \ell\circ F$ for $F\in (I^{-1})_{-d}$ is injective. 

Let $F_1, \dots, F_a$ denote a basis of $U^\perp$ and $x\in R_1$ satisfying property \eqref{eq:Lx}.  Assume 
\[
\sum_{i=1}^nc_ix_i\circ \sum_{j=1}^a a_jF_j=0
\]
for $a_j,c_i\in \kk$, with $a_j\ne 0$ for at least some $j$. We have then 
\[
\sum_{j=1}^a\left(\sum_{i=1}^nc_ix_ia_j \right)\circ F_j=0\,.
\]
It follows that $\left(\sum_{i=1}^nc_ix_ia_j \right)_{1\le j\le a}\in L_{F_1, \dots, F_a}$, and hence $\sum_{i=1}^nc_ix_ia_j$ is a multiple of $x$ for each $j$. Write $x=\sum_{i=1}^nb_ix_i$ with $b_i\in \kk$.  We have thus 
\[
\sum_{i=1}^nc_ix_ia_j=d_j\sum_{i=1}^n b_ix_i
\]
for some $d_j\in \kk$ for all $j\in [a]$. Thus $c_ia_j=b_id_j$ for all $i\in [n]$ and $j\in [a]$. If $j$ is such that $a_j\ne 0$, we have then 
\[
c_i=(a_j)^{-1}d_jb_i \quad\text{ for all $i\in [a]$}
\]
and hence $\sum_{i=1}^nc_ix_i$ is a multiple of $x$. Consequently, the map $(I^{-1})_{-d}\xrightarrow{\ell} (I^{-1})_{-d+1}$ is injective as long as $\ell$ is not a multiple of $x$. 
\end{proof}

\section{Orthogonality in rings of symmetric polynomials}
\label{sec:orthog}
In this section, we keep the notation established in Section 2. We define the rings of invariants $R'$, $S'$, explore choices of bases for the graded components of these rings, and define the map denoted $\perp$ in \cref{eq: CD}. 

\subsection{Partitions} A partition of a positive integer $d$ is a tuple of integers $\l=(\l_1, \ldots, \l_p)$ such that $\l_1\geq \l_2\geq \cdots \geq \l_p>0$ and   $|\l|=\l_1+\cdots+\l_p=d$. We denote the fact that $\l$ is a partition of $d$ by $\l \vdash d$. For a partition $\l$ we write $\#\l$ for the number of parts, not to be confused with $|\l|$, the sum of the parts of $\l$.

If $\l=(\l_1, \ldots, \l_p)\vdash d$ and $p\le n$, we set $x^\l=x_1^{\lambda_1}\cdots x_p^{\lambda_p}$.

If $\l=(\l_1,\dots, \l_p)\vdash d$, we say that a monomial $m$ has {\it type} $\l$ if $m=x_{i_1}^{\l_1}\dots x_{i_p}^{\l_p}$ for some integers $i_1, \dots, i_p$, equivalently, $m=\sigma\cdot x^\l$ for some $\sigma\in \sym_n$. 

We utilize the {\em lexicographic order} on partitions defined as follows: for partitions $\l\vdash d$ and $\mu\vdash d$ we say that $\l>_{\lex}\mu$ if the leftmost nonzero entry of the vector $\l-\mu$ is positive. 

\subsection{Action of the permutation group}
\label{s: action}
We denote by $\sym_n$ the permutation group on $\{1,\dots, n\}$.  
 We consider the action of $\sym_n$ on $R$, respectively $S$, by 
\begin{align*}
\sigma f(x_1, \dots, x_n)&=f(x_{\sigma(1)}, \dots, x_{\sigma(n)}) \qquad\text{for $f\in R$, $\sigma\in \sym_n$}\\
\sigma g(y_1, \dots, y_n)&=g(y_{\sigma(1)}, \dots, y_{\sigma(n)}) \qquad\text{for $g\in S$, $\sigma\in \sym_n$}
\end{align*}
As discussed in \cite[Lemma 3.2, Appendix A]{HSS24}, the action on $S$ is naturally induced from the action on $R$, and it satisfies
\begin{equation}
\label{e: action-RS}
\sigma\cdot (f\circ g)=(\sigma\cdot f)\circ (\sigma\cdot g)
\end{equation}
for $f\in R$, $g\in S$ and $\sigma\in \sym_n$. 
 We denote by $R'$, respectively $S'$, the subring of invariants of $R$, respectively $S'$, under the action of $\sym_n$, namely
\begin{align*}
R'&=\{f\in R\mid \sigma\cdot f=f\quad\text{for all $\sigma\in \sym_n$}\}\\
S'&=\{g\in S\mid \sigma\cdot g=g\quad\text{for all $\sigma\in \sym_n$}\}
\end{align*}

For each $d\ge 0$, we let $P_n(d)$ denote $\dim_{\kk}R'_d$. Observe that $P_n(d)$ is the number of partitions of $d$ with at most $n$ parts. In particular, $P_n(d)$ does not depend on $n$ when $n\ge d$. In this case, we write  $P(d)$ instead of $P_n(d)$. 

When $\text{char}(\kk)>n$ or $\ch\kk=0$, the Reynolds operator is the algebra retract  $\rho\colon R\to R'$ given by 
\[
\rho(f)=\frac{1}{n!} \sum_{\sigma\in \sym_n} \sigma \cdot f \qquad\text{for $f\in R$}.
\]

 With respect to the pairing between $R_d$ and $S_{-d}$ given by $\circ$, the monomials of degree $d$, respectively $-d$, form a pair of canonical dual bases for the two spaces. With respect to the pairing between $R'_d$ and $S'_{-d}$ the choice of dual bases seems less canonical. Given a monomial $\mu$ of type $\l$, we set 
 \begin{eqnarray}
 \label{e: Ml}
 M_\l &=& \rho(\mu)=\frac{1}{n!}\sum_{\sigma\in \sym_n} \sigma\cdot \mu\\
 \label{e: ml}
 m_\l &= &\text{ the sum of all distinct monomials of type }\l. 
 \end{eqnarray}

Let $\l\vdash d$. For each $i\in [d]$ let $p_i$ denote the number of parts of $\l$ of size $i$, and $p_0=n-\#\l$. Then, since the stabilizer under the action of $\sym_n$ of each term in $m_\l$ has size $p_0!p_1!\cdots p_d!$, the symmetric functions above (considered in either $R$ or $S$) are related by 
 \begin{equation}\label{eq: m vs M}
 m_\l=\frac{ n!}{p_0! p_1!\cdots, p_d!} M_\l=\binom{n}{ p_0, p_1,\ldots, p_d} M_\l.
 \end{equation}

Note that $M_\l$ only depends on $\l$ and, in particular, is independent of the choice of $\mu$.
 We intentionally did not specify above the ring in which the monomials are considered, since we will use the same notation when working in both $R$ and $S$, and the appropriate ring will be inferred from the context. For example, in \cref{l:Mm} below one needs to take $M_\l\in R$ and $m_\nu\in S$ in order for the contraction operation to make sense. When working with monomials in $R$,  note that $M_\l,m_\l\in R'$. When working with monomials in $S$, note that $M_\l, m_\l\in S'$.  
 
 \begin{lem}
 \label{l:Mm}
 Assume $\ch\kk>n$ or $\ch \kk=0$. The sets of polynomials $\{m_\l\}_{\l\vdash d}$ and $\{M_\l\}_{\l\vdash d}$ are dual bases for $R'_d$ and $S'_{-d}$ with respect to the pairing $\circ$, that is
 \begin{equation}\label{eq: lem 2.2}
 m_\l\circ M_\nu=
 \begin{cases}
 1 &\text{ if } \mu=\nu\\
 0 &\text{ if } \mu\neq \nu\,.
 \end{cases}
\end{equation}
 \end{lem}
 \begin{proof}
Since $m_\l\circ m_\l$ is equal to the number of terms in $m_\l$, which is given by $n!$ divided by the size of the stabilizer of each term in $m_\l$, we have
\[
m_\l\circ m_\l=\binom{n}{ p_0, p_1,\ldots, p_d}.
\]
By \eqref{eq: m vs M} we deduce that $M_\l\circ m_\l=1$. The remaining case is clear from the definition of the operation $\circ$.
 \end{proof}

The following combinatorial lemma is crucial in the proof of \Cref{prop: dimension linear relations}.

\begin{lem}\label{lem: c_mu_nu}
Let $\mu\vdash (d-1)$. Then
\[
m_\mu \left(\sum_{i=1}^n x_i \right)=\sum_{\l\vdash d} c^\l_\mu m_\l
\]
with $c_\mu^\l\in \kk$ and $c^\l_\mu\neq 0$ if and only if $\l$ can be obtained by adding a box to $\mu$ and re-ordering. When $n\ge d$, the constants $c^\l_\mu$  only depend on $\l$ and $\mu$ and not on $n$.
\end{lem}
\begin{proof}
Indeed, the integer $c^\l_\mu$ counts the number of pairs $(\sigma(x^{\mu}),x_i)$ with $\sigma\in \sym_n$ and $i\in [n]$ so that $x^\l=\sigma(x^{\mu})\cdot x_i$. Since the monomial $x^\l$ only involves at most the first $d$ variables so do $\sigma(x^{\mu})$ and $x_i$. Hence we may assume $\sigma\in \sym_d$ and thus the number $c^\l_\mu$ does not depend on $n$. The last statement follows from the description of $c^\l_\mu$ above. \end{proof}

\subsection{Orthogonal subspaces}
 \cref{eq: CD} uses a map $\perp\colon \Gr(r,R'_d)\to \Gr(P(d)-r, S'_{-d})$, which we now define. 

\begin{notation}\label{not: U perp}
 Let $d\ge 1$, $r\ge 0$ and  let $U\in \Gr(r,R'_d) $. 
 We define a subspace of $S'_{-d}$ associated to $U$ by 
 \[
 U^\perp= \left\{w\in S'_{-d}: u\circ w=0 \text{ for all }u\in U \right \}.
 \]
 Note that $\dim_\kk U^\perp =P(d)-r$, and thus $U^\perp\in \Gr(P(d)-r,S'_{-d})$.

 Conversely, if $W\in \Gr(P(d)-r, S'_{-d})$, we define 
 \[
 W^\perp=\{u\in R'_d\colon u\circ w=0 \text { for all $w\in W$}\}. 
 \]
 Note that $\dim_\kk W^\perp=r$, and thus $W^\perp\in \Gr(r,R'_d)$. Also, we have $(U^\perp)^\perp=U$ and $(W^\perp)^\perp=W$. 
 \end{notation}

This leads to the following double annihilator lemma which shows that taking $R$-annihilators of subspaces of $S'$ corresponds to taking preimages under the Reynolds operator $\rho$.

 \begin{lem} \label{lem: perp map}
Assume $\ch\kk>n$ or $\ch\kk=0$. Using  \cref{not: U perp} and viewing $U^\perp$ as a subspace of $S_{-d}$, we have 
\[
[\Ann_R(U^\perp)]_d=\{f\in R_d\colon \rho(f)\in U\}.
\]
\end{lem}
\begin{proof}
This follows from the observation that if $w\in (S')_{-d}$ then $f\circ w=\rho(f)\circ w$. To see this note that $w=\sum_{\l\vdash d} \theta_\l m_\l$ for some $\theta_\l\in \kk$ since $\{m_\l\}_{\l\vdash d}$ is a basis for $S'_d$ and if $f=\sum_i c_i \mu_i$ with $\mu_i$ monomial of type $\nu_i$, then $\rho(f)=\sum_i c_i M_{\nu_i}$. Now we compare
\begin{eqnarray*}
f\circ w &=& \sum_{i, \l} c_i \theta_\l \mu_i\circ m_\l= \sum_{\nu_i= \l}c_i \theta_\l\\
\rho(f)\circ w &=& \sum_{i, \l} c_i \theta_\l M_{\nu_i}\circ m_\l= \sum_{\nu_i= \l}c_i \theta_\l
\end{eqnarray*}
where we used \eqref{eq: lem 2.2} in the last equality. Therefore we have 
\begin{align*}
f\in [\Ann_R(U^\perp)]_d &\iff  f\circ w=0 \text{ for all } w \in U^\perp\\
&\iff   \rho(f)\circ w=0 \text{ for all } w \in U^\perp\\
&\iff   \rho(f) \in(U^\perp)^\perp =U. \qedhere
\end{align*}
\end{proof}

\section{Linear syzygies}
\label{sec:linear}
Let $n$ and $d$ be positive integers. In this section, we work with the rings $R$, $S$, $R'$ and $S'$ defined in \cref{sec: inverse} and \cref{sec:orthog} and we explore properties of the space of linear syzygies of a subspace of $S_{-d}$. 

\begin{notation}
Let $F_1, \ldots, F_a \in S_{-d}$ be linearly independent forms.  
Recall that the space of linear relations on $F_1, \ldots, F_a$ is 
\[
L_{F_1, \dots, F_a}=\left\{  \left(\sum_{i=1}^nc_{ij}x_i\right)_{1\le j\le a}\!\!\in {R_1}^{a}\colon\, \sum_{j=1}^a \left(\sum_{i=1}^nc_{ij} x_i\right) \circ F_j= 0 \right\}\,.
\]

Consider the following property:
\[
(\mathcal P_{F_1, \dots, F_a}):\,\,\text{If\, $\left(\sum_{i=1}^nc_{ij}x_i\right)_{\!\!1\le j\le a}\!\!\!\!\!\in L_{F_1, \dots, F_a}$, then $c_{ij}=c_{i'j}$ for all $i,i'\in [n]$ and $j\in [a]$.}
\]
\begin{rem}
\label{r: x}
Set $x=x_1+\dots+x_n$. Observe then that the property $(\mathcal P_{F_1, \dots, F_a})$ holds if and only if  
 $L_{F_1,\dots, F_a}\subseteq x(R_0)^a$. Thus, the property $(\mathcal P_{F_1, \dots, F_a})$ establishes \eqref{eq:Lx} in the definition of a $d$-extremely narrow algebra, and it will be used in \cref{sec:narrow-inv} in proving that the algebras of interest are $d$-extremely narrow.  
\end{rem}
\end{notation}

\begin{lem}\label{lem: PW subset}
Let $F_1, \dots, F_b\in S_{-d}$ linearly independent and $G_1, \dots, G_a\in S_{-d}$ linearly independent such that $\Span_\kk( F_1, \dots, F_b)\subseteq \Span_\kk(G_1, \dots, G_a)$. 
If $(\mathcal P_{G_1, \dots, G_a})$ holds, then $(\mathcal P_{F_1, \dots, F_b})$ also holds. 
\end{lem}
\begin{proof}
Assume $(\mathcal P_{G_1, \dots, G_a})$. For each $j\in [b]$ we write 
\[
F_j=\sum_{\ell =1}^a\theta_\ell^j G_\ell
\]
where the vectors $\theta^j := (\theta^j_\ell)_{\ell\in [a]}\in \kk^a$ with $j\in [b]$ are linearly independent. Then 
\[
\sum_{j=1}^b\left(\sum_{i=1}^n{c_{ij}x_i}\right) \circ \sum _{\ell=1}^a\theta^j_\ell G_\ell=\sum _{\ell=1}^a\left(\sum_{i=1}^n\Big(\sum_{j=1}^bc_{ij}\theta_\ell^j\Big)x_i\right)\circ G_\ell
\]
Assume the expression above equals $0$. Then condition $(\mathcal P_{G_1, \dots, G_a})$ gives  \[
\sum_{j=1}^bc_{ij}\theta_\ell^j=\sum_{j=1}^bc_{i'j}\theta_\ell^j\quad\text{for all}\quad i,i'\in [n], \ell\in [a]\,.
\]
Since the vectors $\theta^j$ are linearly independent, it must follow $c_{ij}=c_{i'j}$ for all $i,i'\in [n]$, $j\in [b]$, showing that $(\mathcal P_{F_1, \dots, F_b})$ holds. 
\end{proof}

\begin{rem}
In view of \cref{lem: PW subset}, we see that, if $W$ is a vector subspace of $S_{-d}$ with basis $F_1, \dots, F_b$, then the property $(\mathcal P_{F_1, \dots, F_b})$ is independent of the choice of the basis. In view of this observation, we will write $\mathcal P(W)$ instead of $(\mathcal P_{F_1, \dots, F_b})$. \cref{lem: PW subset} can thus be reformulated as follows: If $W'$ is a subspace of $W$ and $\mathcal P(W)$ holds, then $\mathcal P(W')$ also holds.
\end{rem}

\begin{notation}
Given a subspace $W\subseteq S_{-d}$ we  define
\[
\widetilde{L}_W\coloneqq\left\{ F\in W\colon \sum_{i=1}^n x_i\circ F=0\right\}\,.
\]
\end{notation}

Property $\mathcal P(W)$ states that the linear syzygies of a vector space $W$ are invariant under the action of the symmetric group. We now turn our attention to the case when the subspace $W$ consists of invariant elements under the action of the symmetric group and give sufficient conditions for $\mathcal P(W)$ to hold.

\begin{lem}
\label{lem: 2L}
If $W$ is a subspace of $S'_{-d}$, then 
\begin{equation}
\label{eq: L-inequalities}
\dim_\kk L_W\ge  \dim_\kk \widetilde{L}_W\ge \max\{\dim_\kk W-P_n(d-1),0\}\,.
\end{equation}
and the following statements are equivalent:
\begin{enumerate}
\item The equality $\dim_\kk L_W =  \dim_\kk \widetilde{L}_W$ holds.
\item  $\mathcal P(W)$ holds.
\end{enumerate}
\end{lem}

\begin{proof}
Let $F_1, \dots, F_b$ denote a basis for $W$. Here, we use $L_W=L_{F_1, \dots, F_b}$, noting that $\dim_\kk L_W$ does not depend on the choice of basis. 

We define a map $\eta\colon \widetilde{L}_W\to L_W$
as follows. If $F\in \widetilde{L}_W$, write $F=\sum_{j=1}^b a_jF_j$ and set
\[
\eta(F)=\left(a_j\sum_{i=1}^b x_i\right)_{1\le j\le b}
\]
The equation 
$\sum_{i=1}^n x_i\circ F=0$ from the definition of  $\widetilde{L}_W$ can then be written as 
\[
\sum_{j=1}^b\left(a_j\sum_{i=1}^n x_i\right)\circ F_j=0\,.
\]
and hence $\eta(F)\in L_W$. 
The vector space map $\eta$ is clearly injective and hence 
\begin{equation}
\label{e:Ls}
\dim_\kk \widetilde{L}_W\le \dim_\kk L_W\,,
\end{equation}
with equality if and only if $\eta$ is surjective. Further, note that $\eta$ is surjective if and only if the property $\mathcal P(W)$ holds. This establishes the equivalence (1) $\iff$ (2).

Let $W'$ denote the image of the $\kk$-vector space map $W\to S'_{-(d-1)}$ described by $F\mapsto \sum_{i=1}^n x_i\circ F$. We have an exact sequence 
\[
0\to \widetilde{L}_W\to W\to W'\to 0
\]
 which implies
\[
\dim_\kk \widetilde{L}_W=  \dim_\kk W-\dim_\kk W'\ge \dim_{\kk}W-P_n(d-1)\,.
\]
Using \eqref{e:Ls}, we have then the inequalities in  \eqref{eq: L-inequalities}.  
\end{proof}

Next we give an example of a family of vector spaces $W$ for which $\mathcal P(W)$ holds. To do so, assume $n\ge 2$ and define a property $\mathcal Q(W)$ of a subspace $W\subseteq S'_{-d}$ as follows: 
\[
\mathcal Q(W): \qquad \text{For all $F\in W$, \, $(x_1-x_2)\circ F=0\implies F=0$.}
\]
\begin{notation}
\label{not: delta}
When $n\ge d$, let $\delta$ denote the map that sends $F\in S'_{-d}$ to the polynomial $\delta(F)$ in the variables $y_1, \dots, y_d$  obtained by setting the variables $y_{n}, y_{n-1}, \dots, y_{d+1}$ to $0$.  Note that $\delta$ induces an isomorphism between $S'_{-d}$ and the $(-d)$ component of the invariant ring of $\kk[y_1, \dots, y_d]$, which admits two equivalent descriptions in view of \eqref{eq: m vs M}
\begin{itemize}
\item $\delta$ sends the basis $(m_\l)_{\l\vdash d}$ in $n$ variables to the basis $(m_\l)_{\l\vdash d}$ in $d$ variables,
\item $\delta$ sends the basis $(M_\l)_{\l\vdash d}$ in $n$ variables to $\binom{n}{d}^{-1}(M_\l)_{\l\vdash d}$ in $d$ variables.
\end{itemize}
\end{notation}

Part (3) of \Cref{lem: W0} below generalizes \cite[Proposition 7.7]{HSS24}.

\begin{lem}\label{lem: W0}
Let $W$ be a subspace of $S'_{-d}$. The following then hold: 
\begin{enumerate}
\item If $n\ge d\ge 2$, then $\mathcal Q(\delta(W))$implies $\mathcal Q(W)$; 
\item $\mathcal Q(W)$ implies $\mathcal P(W)$;
\item Assume $\ch \kk>n$ or $\ch \kk=0$. If $\Lambda$  is a set of partitions of $d$ so that $(d)\not \in \Lambda$ and 
$
W_\Lambda=\Span_\kk( M_\l : \l\in \Lambda)
$
then $\mathcal Q(W_\Lambda)$ holds, and hence $\mathcal P(W_\Lambda)$ holds. 
\end{enumerate}
\end{lem}
\begin{proof}
(1) Assume $n\ge d$ and  $\mathcal Q(\delta(W))$ holds. Let $F\in W$ such that $(x_1-x_2)\circ F=0$. We can write $F=\delta(F)+G$, where $G\in S_{-d}$ is such that each monomial in $G$ is divisible by some $y_i$ with $d<i\le n$. 
Then 
\[
0=(x_1-x_2)\circ F=(x_1-x_2)\circ \delta(F)+(x_1-x_2)\circ G\,.
\]
In this sum, $(x_1-x_2)\circ \delta(F)$ is a poynomial in variables $y_1, \dots, y_d$, while $(x_1-x_2)\circ G$ shares the same property as $G$ that each monomial in  $(x_1-x_2)\circ G$ is divisible by some $y_i$ with $i>d$. It follows that $(x_1-x_2)\circ \delta(F)=0$. Since $\mathcal Q(\delta(W))$ holds, we conclude $\delta(F)=0$, implying $F=0$, since $\delta$ is bijective. 

(2) Assume $\mathcal Q(W)$ holds, equivalently  $(x_1-x_2)\circ F\ne 0$ for all nonzero $F\in W$. Since $W\subseteq S'$, \eqref{e: action-RS} gives that $(x_i-x_j)\circ F\ne 0$ for all nonzero $F\in W$ and all $i,j\in [n]$ with $i\ne j$. We consider a basis $F_1, \dots, F_a$ of $W$. 
Assume 
\begin{equation}\label{eq: element of L}
\sum_{\ell=1}^a \left(\sum_{k=1}^n c_{k,\ell}x_k\right)\circ F_\ell=0.
\end{equation}
Let $i,j\in [n]$ distinct and let $\sigma=(ij)\in \sym_n$ be the transposition interchanging $i$ and $j$. Acting with $\sigma$ on \eqref{eq: element of L} and using \eqref{e: action-RS} and the fact that $F_\ell\in S'$ for all $\ell\in [a]$, we have: 
\[
\sum_{\ell=1}^a \left(\sum_{k=1}^n c_{k,\ell}\sigma(x_k)\right)\circ F_\ell=0.
\]
Subtracting this from \eqref{eq:  element of L} gives 
\[
(x_i -x_j)\circ \sum_{\ell=1}^a (c_{i,\ell}-c_{j,\ell})F_\ell=0.
\]
Then $\mathcal Q(W)$ implies that $\sum_{\ell=1}^a (c_{i,\ell}-c_{j,\ell})F_\ell=0$, and hence $c_{i,\ell}=c_{j,\ell}$ for all $\ell\in [a]$ and thus $\mathcal P(W)$ holds.

(3) Since $(d)\notin \Lambda$, observe that a nonzero polynomial $F\in W_\Lambda$ does not contain $(x_k)^d$ in its support for any $k$, and hence $(x_1-x_2)\circ F\ne 0$ by \cite[Lemma 7.6]{HSS24}. Consequently, it follows that $\mathcal Q(W_\Lambda)$ holds and from (2) it follows that $\mathcal P(W_\Lambda)$ holds. 
\end{proof}

\section{Narrow algebras with general invariant socle}
\label{sec:narrow-inv}

In this section, we use the notation in Section 2 for fixed positive integers $n, d,r$ with  $r\le P_n(d)$. We study the parametrization of symmetric ideals in terms of their socle given by the map $\beta$ in \eqref{beta map}.

\begin{defn}
Let $0\le s< P_n(d)$. We say that a quotient $R/I$ is a {\em narrow algebra with $(s,d)$--invariant socle} if the inverse system of $I$ has the form $I^{-1}=W+S_{\geq-d+1}$, where $W$ is a vector subspace of $S'_{-d}$ and $\dim_\kk(W)=s$. 

We parametrize the space $W$ by its  orthogonal complement $U\in \Gr(P_n(d)-s, R'_d)$, where $U^\perp=W$.  We say that a {\it general narrow algebra $A$ with $(s,d)$--invariant socle} satisfies a property $(\mathcal P)$ if there exists a nonempty set $G$ in $\Gr(P_n(d)-s, R'_d)$ such that the property $(\mathcal P)$ holds for $A=R/I$, where $I=\Ann_R(U^\perp+S_{\geq-d+1})$ with $U\in G$. 
\end{defn}

The main result of the section, \Cref{thm: narrow}, says that, under appropriate assumptions on $d$, $s$ and $n$, a general narrow algebra with $(s,d)$--invariant socle is $d$-extremely narrow. For the purpose of later connections, we state the result in terms of  $r=P_n(d)-s$, where $1\le r\le P_n(d)$. 

Recall that $P_n(d)$ is independent of $n$ when $n\ge d$, and we denote it $P(d)$ in this case. Since the dimension of $R'_d$ is equal to $P(d)$, we identify the domain of $\beta$, $\Gr(r,R'_d)$, with $\Gr(r, \kk^{P(d)})$ by identifying the basis $\{m_\l\}_{\l\vdash d}$ with the standard basis of $\kk^{P_n(d)}$.  An important aspect of our result  is uniformity with respect to the number of variables, when $n\ge d$, as implied by the last statement of \cref{thm: narrow} below. 
 
\begin{thm}\label{thm: narrow}
Assume $n,d,r$ are positive integers with $r
\le P_n(d)$, and $\ch \kk>n$ or $\ch \kk=0$. There exists a nonempty open subset  $G$ of $\Gr(r, R'_d)$ such that if $U\in G$, then the ideal $J=\Ann_R(U^\perp+S_{\geqslant -d+1})$ and the algebra $A=R/J$ satisfy the following properties:  
 \begin{enumerate}
 \item The algebra $A$ is narrow with  $\dim_\kk J_d=\dim_\kk R_d-(P_n(d)-r)$. In particular, $J=\m^d$ when $r=P_n(d)$. 
\item If $r<P_n(d)$, then $A$ is $d$-extremely narrow, with socle polynomial 
 \[
 \big(\dim_\kk R_{d-1}-(P_n(d)-r)n+ \max\{P_n(d)-P_n(d-1)-r,0\}\big)z^{d-1}+(P_n(d)-r)z^d\,.
 \]
 \item  $J$ is generated in degree $d$. 
\item The algebra $A$ has the Weak Lefschetz Property. When $r<P_n(d)$,  the map $A_{i-1}\xrightarrow{\ell}A_i$ has maximal rank for all $i$ and all $\ell\in R_1$ such that $\ell$ is not a multiple of $x_1+\dots+x_n$. 
 \end{enumerate}
 Furthermore, when $n\ge d$ the set $G$ is independent of $n$, meaning that, when identifying $\Gr(r,R'_d)$ with $\Gr(r, \kk^{P(d)})$,  the equations defining $G$ (in Pl\"ucker coordinates) have coefficients that do not depend on $n$. 
 \end{thm}

 The proof is given at the end of the section, and it is based on the following proposition, which generalizes \cite[Proposition 7.15]{HSS24} and uses notation introduced in \cref{sec:linear}.  

\begin{prop}\label{prop: dimension linear relations}
Assume the setup of \cref{thm: narrow} and $1\le r\le P_n(d)$.  There exists a nonempty open subset $G$ of $\Gr(r,R'_d)$ such that if $U\in G$ and  $W=U^\perp$, then $\mathcal Q(W)$ and $\mathcal P(W)$ hold and 
\[
\dim_\kk L_{W}=\max\{P_n(d)-P_n(d-1)-r,0\}\,.
\]
Moreover, when $n\ge d$, the set $G$ is independent of the number of variables. 
\end{prop}

\begin{proof}
When $d=1$ or $n=1$ then, since $1\le r\le P_n(d)$, it follows that $r=P_n(d)=1$, and in this case we set $G=\Gr(1,R'_1)$. In this case $W=0$ and $\dim_\kk L_W=0$ as claimed, and $\mathcal Q(W)$, $\mathcal P(W)$ hold vacuously. Assume from now on that $d\ge 2$ and $n\ge 2$. 

For $U\in \Gr(r, R'_d)$, we set  $W=U^\perp$ and  we denote  
\[
C(U)=\dim_\kk \widetilde L_{W}\qquad\text{and}\qquad 
D(U)=\dim_\kk L_{W}\,.
\]
Let $\Lambda_0$ be the set of smallest $P_n(d)-r$ partitions of $d$ with respect to the lexicographic order. Setting
\[
U_0=\langle m_\l :   \l \not\in \Lambda_0 \rangle \subseteq R_{d}
\quad 
\text{yields}
\quad
{U_0}^\perp=\langle M_\l :   \l \in \Lambda_0 \rangle=W_{\Lambda_0} \subseteq S'_{-d}.
\]
We first prove: 
\begin{claim1} 
\label{claim1} There exists a nonempty Zariski open subset $G_1$ of $\Gr(r,R'_d)$ which is independent of $n$ when $n\ge d$,  such that $U_0\in G_1$ and 
\[
C(U)=\max\{P_n(d)-P_n(d-1)-r,0\}\qquad\text{for all}\quad U\in G_1.
\]
\end{claim1}

To prove the claim, let $U\in \Gr(r,R'_d)$, let $W=U^\perp$ and let $F=\sum_{\l\vdash d} \theta_\l M_\l\in S'_{-d}$. 
Then $F\in \widetilde L_W$ if and only if the two conditions below are satisfied
\begin{equation}\label{eq: L tilde conditions}
 \begin{cases}
 F\in U^\perp\\
\left(\sum_{i=1}^n x_i \right)\circ F=0.
 \end{cases}
 \end{equation}
 We analyze the two conditions separately. 
 
 The second condition in \eqref{eq: L tilde conditions} is equivalent to the system of linear equations
 \begin{equation*}
\left( m_\mu \sum_{i=1}^n x_i\right) \circ F=0 \text{ for all } \mu \vdash d-1.
\end{equation*}
Utilizing \Cref{lem: c_mu_nu}, including the notation $c^\l_\mu$ introduced therein,   we express the products $m_\mu \sum_{i=1}^n x_i$ in the basis $\{m_\l\}$ of $R'_d$ as follows
\[
\left(\sum_{\l\vdash d} c^\l_\mu m_\l\right )\circ \left ( \sum_{\nu\vdash d} \theta_\nu M_\nu \right)= \sum_{\l\vdash d} c^\l_\mu \theta_\l =0 \quad \text{ for all } \mu \vdash d-1.
\]
The set of equations above is equivalent to the following system of equations, which is independent of $n$ when $n\ge d$: 
\begin{equation}\label{eq: symmetric system 2}
\sum_{\l\vdash d} c^\l_\mu \theta_\l =0 \quad \text{ for all } \mu \vdash d-1.
\end{equation}

Recall from \Cref{lem: c_mu_nu} that $c^\l_\mu\neq 0$ if and only if $\l$ can be obtained by adding a box to $\mu$ and possibly re-ordering. In particular, this implies that if $c^\l_\mu\neq 0$ then $\l\geq_{\lex}\mu_{\square}$, where lex refers to the lexicographic order and $\mu_{\square}\vdash d$ is the partition $\mu$ followed by an additional part equal to 1. Moreover, if $\mu$ is the $i$-th smallest partition of $d-1$ in the lexicographic order then $\mu_{\square}$ is the $i$-th smallest partition of $d$ in the lexicographic order. Thus, arranging the partitions of $d$ and $d-1$ in lexicographic order from smallest to largest, this implies that the $P_n(d-1)\times P_n(d)$ matrix with scalar entries $c_\l^\mu$ of the system \eqref{eq: symmetric system 2} is in row echelon form with pivots in every row, occupying the first $P_n(d-1)$ columns.  In particular the rank of the matrix $\begin{bmatrix}  c^\l_\mu\end{bmatrix}$ is $P_n(d-1)$.

Next we turn to the first condition in \eqref{eq: L tilde conditions}, namely that $F\in U^\perp$. Consider a basis for $U$ consisting of $r$ polynomials $u_i=\sum_{\l \vdash d} u_{i,\l} m_\l$. Due to \eqref{eq: lem 2.2}, the condition $F\in U^\perp$ becomes
\begin{equation}\label{eq: symmetric system 1}
\sum_{\l\vdash d} u_{i,\l}\theta_\l =0 \quad \text{ for all } 1\leq i\leq r.
\end{equation}

Upon combining the equations \eqref{eq: symmetric system 2} and \eqref{eq: symmetric system 1}, we argue next that the corresponding $(P_n(d-1)+r)\times P_n(d)$ matrix 
\[
A_U=\begin{bmatrix}
c^\l_\mu\\
 u_{i,\l}\
\end{bmatrix}_{\mu\vdash d-1, i=1,\ldots,r , \l \vdash d}
\]
attains maximal rank $\max\{P_n(d-1)+r,P_n(d)\}$, provided that $U$ is chosen in an appropriate open set of $\Gr(r,R'_d)$. First, note that if $U=U_0$, then the matrix formed by the last $r$ rows of $A_{U_0}$, coming from \eqref{eq: symmetric system 1}, is in echelon form, with pivots in the last $r$ columns. Recalling that the matrix formed by the first $P_n(d-1)$ rows of $A_{U_0}$ is in echelon form with pivots in the first $P_n(d-1)$ columns, we conclude that $A_{U_0}$ has rank $P_n(d-1)+r$ provided that $P_n(d-1)+r$ does not exceed $P_n(d)$, and $P_n(d)$ otherwise.  

The condition  that the matrix $A_U$ does not have maximal rank is the vanishing of its maximal minors. The maximal minors of $A_U$  can be written by generalized Laplace expansion as linear combinations of the maximal minors of $\begin{bmatrix}  u_{i,\l}\end{bmatrix}$ with coefficients given by complementary  minors of $\begin{bmatrix}  c^\l_\mu\end{bmatrix}$. That is, the locus where $A_U$ fails to have maximal rank is cut out by linear equations in the Pl\"ucker coordinates of $U$. Therefore, the complement $G_1$ of the set cut out by these equations is open and $U_0\in G_1$ by the previous considerations. The set $G_1$ does not depend on $n$ when $n\ge d$, since the entries of $A_U$, which determine the complement of $G_1$, are independent of $n$.

 Since $A_U$ has maximal rank on the set $G_1$, we conclude that \[
 C(U)=P_n(d)-\rank A_U=\max\{P_n(d)-P_n(d-1)-r,0\}
 \]
 for all $U\in G_1$.  
 
We now show: 
 \begin{claim2} There exists a nonempty  open set $G_2$ of $\Gr(r,R'_d)$, which is independent of $n$ when $n\ge d$, such that $\mathcal Q(U^\perp)$ holds for all $U\in G_2$.
 \end{claim2} 
   
 To prove Claim 2, consider the following two conditions on $F=\sum_{\l\vdash d} \theta_\l M_\l\in S'_{-d}$:
 \begin{align}
 \begin{split}
 \label{eq: 12F}
 (x_1-x_2)&\circ F=0\qquad\text{and}\\
 F&\in U^\perp\,,
 \end{split}
 \end{align}
 where $U$ has a basis consisting of $u_i=\sum_{\l \vdash d} u_{i,\l} m_\l$, with $1\le i\le r$. For a given $U$, the condition $\mathcal Q(U^\perp)$ holds if and only if the equations \eqref{eq: 12F} imply $F=0$. As seen above, the condition $F\in U^\perp$ translates as the equations \eqref{eq: symmetric system 1}. Further, upon writing the expression $(x_1-x_2)\circ F$ in terms of the monomial basis of $S_{-d+1}$ and equating the coefficients to 0, the condition $(x_1-x_2)\circ F=0$ translates into $\dim_{\kk} S_{-d+1}$ linear equations in the variables $\theta_\l$. We denote by $C_U$ the matrix of this system. Thus, \eqref{eq: 12F} translates into a homogeneous system of $\dim_\kk S_{-d+1}+r$ equations in $P_n(d)$ indeterminates. 
 Using \Cref{lem: W0}(3) we deduce that $\mathcal Q(W)$ holds for $W=W_{\Lambda_0}={U_0}^\perp$, meaning that, when $U=U_0$, the system \eqref{eq: 12F} has only the trivial solution. In general, for $U\in \Gr(r,R'_d)$, the condition that the system has only the trivial solution is equivalent to the fact that its matrix has maximal rank. This matrix has the form
 \[
B_U=\begin{bmatrix}
C_U\\
 u_{i,\l}\
\end{bmatrix}_{i=1,\ldots,r , \l \vdash d}.
\]
As observed earlier, when $U=U_0$ the matrix formed by the last $r$ rows of $B_{U_0}$ is in row echelon form, with pivots in the last $r$ columns, and hence there exists a maximal nonzero minor $\mathcal M_0$ of $B_{U_0}$ that uses these last $r$ rows. For any $U$, the condition that the maximal minor of $B_U$ using the same rows and columns as $\mathcal M_0$
is zero is a linear equation in the Pl\"ucker coordinates of $U$. The complement $G(n)$ of the set cut out by this equation is thus open, and contains $U_0$. We conclude that there exists a nonempty open set $G(n)$  for which \eqref{eq: 12F} implies $F=0$ whenever $U\in G(n)$. In view of \cref{lem: W0}(2), we see that $\mathcal P(U^\perp)$ holds for all $U\in G(n)$. The set $G(n)$ depends on $n$ because the matrix $C_U$ depends on $n$. However, when $n\ge d$,  we  describe next a nonempty open set that does not depend on $n$, while satisfying the same properties. 

Assume $n\ge d$, and hence $P_n(d)=P(d)$. Since the bijection that takes an $r$-dimensional subspace of $R'_d$  to its $(P(d)-r)$-dimensional orthogonal complement in $S'_{-d}$ induces an isomorphism of algebraic varieties between the respective Grassmannians, the set $G(d)^\perp=\{U^\perp \mid U\in G(d)\}$ is an open subset of the Grassmannian of $(P(d)-r)$-dimensional subspaces of degree $-d$ invariant polynomials in $\kk[y_1, \dots, y_d]$. We parametrize $G(d)$ using Pl\"ucker coordinates corresponding to the basis $\{m_\lambda\}_{\lambda\vdash d}$ and $G(d)^\perp$ using the  Pl\"ucker coordinates corresponding to the dual basis $\{M_\lambda\}_{\lambda\vdash d}$, cf. \eqref{l:Mm}. We use the map $\delta$ introduced in \cref{not: delta}. Since $\delta(M_\l)=\binom{n}{d}^{-1}M_\l$, where the first $M_\l$ is a  polynomial in $n$ variables and the second is a  polynomial in $d$ variables, $\delta$ multiplies each Pl\"ucker coordinate by $\binom{n}{d}^{-P(d)+r}$ and thus induces the identity map on tuples of Pl\"ucker coordinates viewed as coordinates in projective space.

Let $D$ denote the preimage of  $G(d)^\perp$ under $\delta$. Then $D$ is an open subset of the Grassmannian $\Gr(P(d)-r, S'_d)$. While the elements of $S'_{-d}$ depend on $n$, the set $D$ does not, as the coefficients of its defining equations are the same as those of the set $G(d)^\perp$, which does not depend on $n$.  We set then $G_2=D^\perp$. If $U\in G_2$, then $U^\perp\in D$ and hence $\delta(U^\perp)\in G(d)^\perp$, meaning that $(\delta(U^\perp))^\perp\in G(d)$. By the definition of $G(d)$, we know that $\mathcal Q(\delta(U^\perp))$ holds, and hence $\mathcal Q(U^\perp)$ holds by \cref{lem: W0}(1).  This finishes the proof of Claim 2.

It follows from Claim 2 and \cref{lem: W0} that $\mathcal P(U^\perp)$ also holds for all $U\in G_2$. Then, \cref{lem: 2L} gives that $C(U)=D(U)$ for all $U\in G_2$. Using also Claim 1, we conclude
 \[
 D(U)=\max\{P_n(d)-P_n(d-1)-r,0\}\qquad\text{for all $U\in G_1\cap G_2$.}
 \]
 which gives the conclusion of our result.
\end{proof}

We now give the proof of the theorem. 

 \begin{proof}[\it Proof of \cref{thm: narrow}]
 Let $G$ be as in \cref{prop: dimension linear relations}.
  
Part (1) is a direct consequence of \cref{prop: narrow}, noting that $\dim_\kk U^\perp=P_n(d)-r$. This statement holds for any $U\in \Gr(r,R'_d)$, and does not require $U\in G$. 
 
For part (2), let $U\in G$ and set $A=R/J$, where $J=\Ann_R(U^\perp+S_{\geqslant -d+1})$, and assume $r<P_n(d)$. 
 By \cref{prop: dimension linear relations}, we know that $\mathcal P (U^\perp)$ holds and 
 \begin{equation}
 \label{e:dimL}
\dim_\kk L_{U^\perp}=\max\{P_n(d)-P_n(d-1)-r,0\}\,.
\end{equation}
 In view of \cref{r: x}, we conclude that \eqref{eq:Lx} holds with $x=x_1+\dots+x_n$, hence $A$ is $d$-extremely narrow when $n\ge 3$ by \cref{prop: narrow}. When $n<3$, \eqref{e:dimL} gives $\dim_\kk L_{U^\perp}=0$, and hence $A$ is $d$-extremely narrow as well, by \cref{rem:L=0}. 
 
Part (3) follows from part (1) and part (2). Part (4) follows from \cref{p: e-narrow-has-WLP} when $r<P_n(d)$, and is clear when $r=P_n(d)$.
\end{proof}

\section{Defining equations of a narrow algebra with general invariant socle}
\label{s: equations}

In this section, our goal is to find generators up to symmetry for the symmetric ideal $\Ann_R(U^\perp+S_{\geqslant -d+1})$, where $U$ is in the open set of \cref{thm: narrow} and $n$ is sufficiently large. In \Cref{t: CD} we will establish that these ideals are the general symmetric ideals, thus our results in this section ultimately describe the generators of general $(r,d)$-symmetric ideals in sufficiently large number of variables. 
This is established in \cref{thm: gens} below, whose proof ingredients revise and extend some of the arguments in \cite{HSS24}. A notable aspect of our results is that the generators do not depend on $n$, as long as $n$ is sufficiently large. 

First we recall some terminology introduced in \cite{HSS24}.

\begin{defn}
\label{d:admissible} Fix a partition $\lambda=(\lambda_1,\cdots,\lambda_s)$ of $d$ with $s$ parts, with $d\ge 1$. 

If $m$ is a monomial of degree $d$, then we write $\type(m)=\lambda$ if $\l$ is the unique partition of $d$ such that $m$ has type $\l$. We denote by $()$ the  empty partition with $0$ parts and set $\type(1)=()$. 

We say that $b\in R$ is a $\lambda$-{\it binomial} if $b=m-m'$, with $m,m'$  distinct monomials of type $\lambda$. For $b$ as above, we set $g(b)=\gcd(m,m')$ and we say that $b$ is {\it admissible} if $g(b)$ is relatively prime with both $\frac{m}{g(b)}$ and $\frac{m'}{g(b)}$. 
\end{defn}

\begin{ex}
\label{ex: admissible}
 Let $d=3$ and $\lambda=(2,1)$. Then the $\lambda$-binomials $x_1^2x_2-x_1^2x_3$, $x_1^2x_2-x_3^2x_2$, $x_1^2x_2-x_3^2x_4$ are all admissible, while the binomial $x_1^2x_2-x_2^2x_3$ is not admissible. 
\end{ex}

\begin{rem}
\label{rem: decompose-admissible}
A $\l$-binomial is either admissible, or else it can be written as a sum of admissible $\l$-binomials. 
For example, consider the decompositions
\begin{align*}
x_1^2x_2-x_2^2x_3&=(x_1^2x_2-x_1^2x_3)+(x_1^2x_3-x_2^2x_3)\\
x_1^3x_2^2x_3-x_3^3x_1^2x_2&=(x_1^3x_2^2x_3-x_4^3x_5^2x_6)+(x_4^3x_5^2x_6-x_3^3x_1^2x_2)\,.
\end{align*}
In fact, when $n$ is sufficiently large, this decomposition can be done with only two admissible $\l$-binomials, as in the second equation above. 
\end{rem}

\begin{defn}
 We say that a partition $\gamma$ is a {\it subpartition} of the partition $\lambda$, and we write $\gamma\subseteq \lambda$, if the multiset of parts of $\gamma$ is a submultiset of the parts of $\lambda$. If $\gamma\subseteq\lambda$, we describe $\gamma$ by indicating which parts of $\lambda$ are present in $\gamma$ as follows. 
 If $\gamma=(\gamma_1, \cdots, \gamma_t)$, then we can choose distinct integers $k_1, \cdots, k_t$ such that  $\Psi_i=\lambda_{k_i}$ for all $i$ with $1\le i\le t$.  Since the choice of such integers is not necessarily unique, we further require that these indices are chosen in order, starting with $k_1$ and ending with $k_t$ and, whenever a choice is to be made, we choose $k_i$ to be the smallest of the available choices. The indices $k_i$ are unique and we set $T(\lambda, \gamma) :=\{k_1, \dots, k_t\}$. 
\end{defn}

\begin{ex}
\label{e:div}
If $\lambda=(5,5,2,2,1)$, then $\gamma=(5,5,2,1)$ is a subpartition of $\l$. There are two choices for $k_1$ with $\gamma_1=\lambda_{k_1}$, namely $k_1=1$ or $k_1=2$. Since we must choose the smallest of the choices,  we  take $k_1=1$. The only remaining choice for $k_2\ne k_1$ with $\gamma_2=\lambda_{k_2}$ is $k_2=2$. Further, there are two choices for $k_3$ with  $\gamma_3=\lambda_{k_3}$, namely $k_3=3$ or $k_3=4$.  Our definition dictates $k_3=3$. Finally, we see that $\gamma_4=1=\l_5$, and hence $k_4=5$, and thus $T((5,5,2,2,1),(5,5,2,1))=\{1,2,3,5\}$.
\end{ex}

\begin{ex}
    Continuing on \cref{ex: admissible} above, notice that the type of $g(b)$ with $b=x_1^2x_2-x_1^2x_3$, $x_1^2x_2-x_3^2x_2$, respectively, $x_1^2x_2-x_3^2x_4$ is $\gamma=(2)$, $(1)$, respectively $()$. These are all subpartitions of $\l=(2,1)$ not equal to $\l$, and $T(\lambda, \gamma)=\{1\}$, $\{2\}$, respectively $\emptyset$. 
\end{ex}

\begin{rem}
Let $d\ge 1$ and let  $\lambda=(\lambda_1, \dots, \lambda_s)$ be a partition with $s$ parts. Let $b$ be a $\lambda$-binomial and let $\gamma=\text{type}(g(b))$.
The following statements are  then equivalent:
\begin{enumerate}
\item $b$ is admissible; 
\item $\gamma\subsetneq \lambda$, and there exists a choice of indices with  $i_1, \dots, i_s$ distinct and $j_1, \dots, j_s$ distinct such that  $b=x_{i_1}^{\lambda_1} \cdots x_{i_s}^{\lambda_s} - x_{j_1}^{\lambda_1} \cdots  x_{j_s}^{\lambda_s}$ and
\begin{equation}
\label{e:T}
i_\ell=j_\ell\quad\text{for all $\ell\in T$}\qquad\text{and}\qquad  \{i_\ell \, \mid \, \ell \not \in T\} \cap \{j_\ell \, \mid \, \ell \not \in T\} = \emptyset
\end{equation}
where $T=T(\l,\gamma)$. 
\end{enumerate}
\end{rem}

 The following is a modification of the construction in \cite[Construction 6.8]{HSS24}.
\begin{constr}
 \label{c:f} 
Let $\kk$ be a field.  Fix  a positive integer $d$  and ${\boldsymbol{\alpha}}=(\alpha_\l)\in \kk^{P(d)}$. We define
 \begin{equation}
 \label{e:h}
 h_{{\boldsymbol{\alpha}}}=\sum_{\l\vdash d }\alpha_\l x^\l\,.
 \end{equation}
 Fix an order on the partitions $\l\vdash d$ (e.g. lexicographic order). 
 Assume $\bsa\ne 0$ and let $\tau\vdash d$ with $\tau=(\tau_1, \dots, \tau_s)$ be smallest with $\alpha_\tau\ne 0$ in the order we fixed.  
Assume $n$ is sufficiently large for the construction below to be possible. For each partition $\l\vdash d$, $\l\ne \tau$ and for each partition $\gamma$ with $\gamma\subsetneq \lambda$,  including $\gamma=()$, we choose an admissible  $\lambda$-binomial $b(\lambda, \gamma)$ with $g(b(\l, \gamma))=\gamma$ 
and such that none of the variables in the support of $b(\l,\gamma)$ are in the support of $h_{ {\boldsymbol{\alpha}}}$ and the supports of any two binomials  of the form $b(\lambda, \gamma)$  are disjoint.  We set
\begin{equation}
\label{e:f}
f_{\bsa}\coloneq h_{{\boldsymbol{\alpha}}}+\sum_{\lambda\vdash d, \lambda\neq \tau}\, \sum_{\gamma\subsetneq \lambda} b(\lambda, \gamma)\,.
\end{equation}
Observe that the smallest number of variables needed for this construction to be possible for all choices of $\boldsymbol{\alpha}$ depends only on $d$. 
\end{constr}

\begin{ex}
Let $d=3$ and assume $\tau=(1,1,1)$ for a given $\bsa\in \kk^3$. Then 
\begin{align*}
 f_{ {\boldsymbol{\alpha}}}=\alpha_{(1,1,1)}x_1x_2x_3+\alpha_{(2,1)}x_1^2x_2&+\alpha_{(3)}x_1^3+(x_4^3-x_5^3)+(x_6^2x_7-x_8^2x_9)+\\&+(x_{10}^2x_{11}-x_{10}^2x_{12}) (x_{13}^2x_{14}-x_{15}^2x_{14})\,.
 \end{align*}
\end{ex}

\begin{rem}
Recall that, if $m$ is a monomial of type $\l$, then $\rho(m)=M_\l$. It follows that if $b$ is a $\l$-binomial, then $\rho(b)=0$. Further, we have  
\begin{equation}
\label{eq: alpha}
\rho(f_{\bsa})=\rho(h_{\bsa})=\sum_{\l\vdash d} \alpha_\l M_\l
\end{equation}
\end{rem}

\begin{notation}
\label{not: sigma-m}
    For each monomial $m=x_{i_1}^{\tau_1}x_{i_2}^{\tau_2}\dots x_{i_s}^{\tau_s}$ with $\text{type}(m)=\tau$, let  $\sigma_{m}\in \sym_n$ denote a fixed permutation with  
    $\sigma_m(j)= i_j$ for $j\in[s]$, 
    so that $\sigma_m\cdot x^\tau=m$. 

    For any permutation $\l\vdash d$, we let $\mathcal N_\l$ denote the set of all monomials of type $\l$. 
\end{notation}

The next result is a modified version of \cite[Proposition 6.12]{HSS24} and uses \cref{not: sigma-m} and the notation in \eqref{e: Ml}--\eqref{e: ml}.  
\begin{prop}
\label{p:new-example}
Let $\kk$ be a field with ${\rm char}(\kk )\neq 2$. Let $d,r$ be positive integers.  Let $n$ be a positive integer that is  sufficiently large, as needed for \cref{c:f}. 

Let ${\boldsymbol{\alpha}}=(\alpha_\l)_{\l\vdash d}\in \kk^{P(d)}$ nonzero, and let $f_{\bsa}\in R = \kk[x_1,\cdots, x_n]$ and $\tau\vdash d$ as in \cref{c:f},
 and set  $I=(f_{\bsa})_{\sym_n}$. The following hold:
\begin{enumerate}
\item The principal symmetric ideal $I$ generated by $f_{\bsa}$ may be computed as follows: 
\begin{equation}
\label{eq:gens large f}
I=\sum_{m\in \mathcal N_\tau}(\sigma_m\cdot h_{{\boldsymbol{\alpha}}})+ \sum_{\lambda\vdash d, \lambda\neq \tau}\,\sum_{\gamma\subsetneq \lambda} \left(b(\lambda, \gamma)\right)_{\sym_n}
\end{equation}
where $\mathcal N_\tau$ denotes the set of all monomials of type $\tau$. 
\item The minimal number of generators of $I$ is given by
\[
\dim_\kk I_d= \dim_\kk R_d-(P(d)-1)\,.
\]
\item The component of the Macaulay inverse system  of $I$ in degree $(-d)$ can be computed as 
\[
(I^{-1})_{-d}=\Span_\kk(\alpha_\tau m_\l-\alpha_\l m_\tau \mid \l\vdash d, \l\neq \tau).\]
If $\ch \kk>n$ or $\ch \kk=0$, we also have
\[
(I^{-1})_{-d}=\left(\Span_\kk\sum_{\l\vdash d} \alpha_\l M_\l\right)^\perp
\]
\end{enumerate} 
\end{prop}

 \begin{proof} 
We begin with (1). We first show 
\begin{equation}
\label{eq:gens I}
I=(h_{{\boldsymbol{\alpha}}})_{\sym_n}+ \sum_{\lambda\vdash d, \lambda\neq \tau}\,\sum_{\gamma\subsetneq \lambda} \left(b(\lambda, \gamma)\right)_{\sym_n}.
\end{equation}
Since the LHS of~\eqref{eq:gens I} clearly lies in the RHS, it suffices to show the opposite inclusion, namely, that the RHS of~\eqref{eq:gens I} lies in the LHS. In particular, 
it suffices to show that each summand in the RHS belongs to $I$.  To see this, fix a partition $\l\vdash d$, $\l\ne (d)$ with $s$ parts and $\gamma\subsetneq \lambda$ and let 
\[
b(\lambda, \gamma)=x_{i_1}^{\lambda_1} \cdots x_{i_s}^{\lambda_s} - x_{j_1}^{\lambda_1} \cdots  x_{j_s}^{\lambda_s}\,.
\]
such that $T=T(\lambda, \gamma)$ satisfies \eqref{e:T}. In particular, $i_t=j_t$ for $t\in T$ and  $i_t\ne j_t$ for $t\notin T$. Define the permutation
\[\sigma :=\prod_{t\not \in T}(i_t,  j_t)\,.\] 
By the assumptions on $T=T(\l,\gamma)$ in \eqref{e:T}, it follows that the transpositions $(i_t,j_t)$  commute with each other, and thus the multiplication order in $\sigma$ is irrelevant. 
By construction of $f$, each variable that appears in $b(\lambda, \gamma)$ does not appear in any $b(\lambda',\gamma')$ with $(\lambda, \gamma)\ne (\lambda',\gamma')$ or in $h_{{\boldsymbol{\alpha}}}$, and it follows that $\sigma$ fixes $b(\l',\gamma')$ and $h_{{\boldsymbol{\alpha}}}$. On the other hand, observe that 
\[
\sigma\cdot b(\l,\gamma)=-b(\l,\gamma)\,.
\]
 From this we conclude
\[
b(\l,\gamma)=\frac{1}{2}\left(f_{ {\boldsymbol{\alpha}}}-\sigma \cdot f_{ {\boldsymbol{\alpha}}} \right)\in I.
 \]
 Since $I$ is $\sym_n$-invariant, if one binomial is contained in $I$ then any permutation of it also lies in $I$, so we conclude $(b(\lambda,\gamma))_{\sym_n}$ is contained in $I$, as desired. Moreover, by subtracting binomials from $f_{ {\boldsymbol{\alpha}}}$, we may also conclude that $h_{{\boldsymbol{\alpha}}}\in I$, completing the proof of \eqref{eq:gens I}.  To complete the proof of \eqref{eq:gens large f}, it suffices to show
 \[
 (h_{{\boldsymbol{\alpha}}})_{\sym_n}\subseteq \sum_{m\in \mathcal N_\tau}(\sigma_m\cdot h_{\bsa})+\sum_{\lambda\vdash d, \lambda\neq \tau}\,\sum_{\gamma\subsetneq \lambda} \left(b(\lambda, \gamma)\right)_{\sym_n}\,.
 \]
 We need to show that $\sigma\cdot h_{ {\boldsymbol{\alpha}}}$ is in RHS for all $\sigma\in \sym_n$. Indeed, let $\sigma\in \sym_n$ and set $m=\sigma\cdot x^\tau$. Since $\sigma\cdot x^\tau=\sigma_m\cdot x^\tau$ and in view of \cref{rem: decompose-admissible},  we have 
 \begin{align*}
 \sigma\cdot h_{{\boldsymbol{\alpha}}}-\sigma_m\cdot h_{{\boldsymbol{\alpha}}}&=\sum_{\l\vdash d}\alpha_\l \left(\sigma\cdot x^\l-\sigma_m\cdot x^\l \right)=\sum_{\l\vdash d, \l\ne \tau}\alpha_\l \left(\sigma\cdot x^\l-\sigma_m\cdot x^\l \right)
 \\
 & \in \sum_{\lambda\vdash d, \lambda\neq \tau}\,\sum_{\gamma\subsetneq \lambda} \left(b(\lambda, \gamma)\right)_{\sym_n}\,.
 \end{align*}
 This finishes the proof of (1). 

We now prove (2). For each $\l\vdash d$, we let $\mathcal M_\l$ denote the $\kk$-span of all monomials of type $\l$. 
We have
\[
\dim_\kk\left( \sum_{\scriptscriptstyle\gamma\subsetneq \lambda} \left( b(\l,\gamma)\right)_{\sym_n}\right)_d=\dim_\kk \mathcal M_\l -1.
\]
For a justification of this identity, consult the proof of \cite[Proposition 6.12]{HSS24}. Since the argument is identical to the one used there, we do not repeat it here. 

Set $C\coloneq\left(\sum_{\lambda\vdash d, \lambda\neq \tau}\,\sum_{\gamma\subsetneq \lambda} \left(b(\lambda, \gamma)\right)_{\sym_n}\right)_d$ and $D\coloneq\sum_{m\in \mathcal N_\tau}\langle \sigma_m\cdot h_{{\boldsymbol{\alpha}}}\rangle$. 

Since $R_d=\bigoplus_{\lambda \vdash d} \mathcal M_\lambda$, we have: 
\begin{align}
\label{eq:part}
\begin{split}
\dim_\kk C=&\sum_{\lambda\vdash d, \lambda\neq \tau} (\dim_\kk \mathcal M_\l-1)\\&=\left(\sum_{\lambda\vdash d, \lambda\neq \tau} \dim_\kk \mathcal M_\l\right) -(P(d)-1)\\&=\dim_\kk R_d-\dim_\kk \mathcal M_\tau-(P(d)-1)\,.
\end{split}
\end{align}
Further, using \eqref{e:h} observe that the elements $\sigma_m\cdot h_{ {\boldsymbol{\alpha}}}$ with $m\in \mathcal N_\tau$ are linearly independent, since each element $\sigma_m\cdot h_{ {\boldsymbol{\alpha}}}$ contains in its support exactly one element of type $\tau$, namely $m$. Thus, the span $D$ of these elements has dimension equal to $|\mathcal N_\tau|=\dim_\kk \mathcal M_\tau$, and $C\cap D=0$, since all nonzero elements of $D$ contain in their support a monomial of type $\tau$, while a nonzero element of $C$ does not contain any monomial of type $\tau$ in its support. 
Consequently, the decomposition $I_d=C+D$ given by \eqref{eq:gens large f} is a direct sum. Using \eqref{eq:part}, we obtain
\[
\dim_\kk I_d=\left(\dim_\kk R_d-\dim_\kk \mathcal M_\tau-(P(d)-1)\right)+\dim_\kk \mathcal M_\tau=\dim_\kk R_d-(P(d)-1)\,.
\]

Finally, we prove (3). The inclusion $\Span_\kk(\alpha_\tau m_\l-\alpha_\l m_\tau \mid \l\vdash d, \l\neq \tau) \subseteq  (I^{-1})_{-d}$ can be seen as follows.  In view of~\eqref{eq:gens large f}, it suffices  to show $h_{ {\boldsymbol{\alpha}}}\circ (\alpha_\tau m_\l-\alpha_\l m_\tau)=0$ and $b(\l,\gamma)\circ (\alpha_\tau m_\l-\alpha_\l m_\tau)  =0$ for all $\l, \gamma$ with $\l\vdash d$, $\l\ne \tau$ and $\gamma\subsetneq \lambda$. Both these equations follow from the following observation: 
If $m\in R$ is a monomial of type $\lambda'$ for some $\lambda'\vdash d$, then 
$$
m\circ m_\lambda =\begin{cases}
0 &\text{if $\lambda'\ne \lambda$}\\
1 &\text{if $\lambda'=\lambda$.}
\end{cases}
$$
Next we wish to show the containment $(I^{-1})_{-d} \subseteq  \left\langle \alpha_\tau m_\l-\alpha_\l m_\tau \mid \l\vdash d, \l\neq \tau\right\rangle $. To this extent, it suffices to show that the two vector spaces have the same dimension.  Indeed, using part (2) and observing that the elements $\alpha_\tau m_\l-\alpha_\l m_\tau$ with $\l\vdash d$, $\l\ne\tau$ are linearly independent, we have
\begin{align*}
\dim_\kk(I^{-1})_{-d}&=\dim_\kk(R/I)_d=\dim_\kk R_d-\dim_\kk I_d=P(d)-1\\
&=\dim_\kk\Span_\kk(\alpha_\tau m_\l-\alpha_\l m_\tau \mid \l\vdash d, \l\neq \tau).
\end{align*}

Assume now $\ch \kk>n$ or $\ch \kk=0$. Since $\alpha_\tau\ne 0$, observe that 
 \begin{equation*}\label{eq: def W} 
 \left(\Span_\kk\sum_{\l\vdash d} \alpha_\l M_\l \right)^\perp=\Span_\kk\left(\alpha_\tau m_\l-\alpha_\l m_\tau \mid \l\vdash d, \l\neq \tau\right). \qedhere
 \end{equation*} 
\end{proof}

\begin{prop}
\label{p:from W to f}
Let $d,r$ be positive integers with $1\le r\le P(d)$. Assume $n$ is sufficiently large so that \cref{c:f} can be done and $\ch\kk>n$ or $\ch \kk=0$. 
If $\bsa^1, \dots, \bsa^r$ are  elements in $\kk^{P(d)}$ so that \[
U=\Span_\kk\left(\sum_{\l\vdash d}\bsa^1_\l M_\l, \ldots,  \sum_{\l\vdash d}\bsa^r_\l M_\l\right) \in \Gr(r,R'_{d}),
\]
then 
\[
\left[(f_{\bsa^1}, \dots, f_{\bsa^r})_{\sym_n}\right]_d=\left[\Ann_R(U^\perp)\right]_d\,.
\]
\end{prop}
\begin{proof}
Using \cref{p:new-example} for the third equality below, we have: 
\begin{align*}
[\Ann_R(U^\perp)]_d&=\left[\Ann_R \, \left(\bigcap_{i=1}^r\left(\Span_\kk \sum_{\l\vdash d}\bsa^i_\l M_\l \right )^\perp\right)\right]_d\\
&=\sum_{i=1}^r \left[\Ann_R\left(\left (\Span_\kk  \sum_{\l\vdash d}\bsa^i_\l M_\l \right)^\perp \right)\right]_d\\
&=\sum_{i=1}^r [(f_{\bsa^i})_{\sym_n}]_{d}\\
&=[(f_{\bsa^1}, \dots, f_{\bsa^r})_{\sym_n}]_d\,.\qedhere
\end{align*}
\end{proof}

\begin{thm}
\label{thm: gens}
Let $d$, $r$ be positive integers with $1\le r\le P(d)$ and assume $n$ is sufficiently large so that \cref{c:f} can be done, and assume $\ch\kk>n$ or $\ch \kk=0$. Let $U\in \Gr(r, R'_{d})$ with 
\[
U=\Span_\kk\left(\sum_{\l\vdash d}\bsa^1_\l M_\l, \ldots,  \sum_{\l\vdash d}\bsa^r_\l M_\l \right) 
\]
where $\bsa^i=(\bsa^i_\l)_{\l\vdash d}\in \kk^{P(d)}$. Set $V=\Span_\kk( f_{\bsa^1}, \dots, f_{\bsa^r})$, with $f_{\bsa^i}$ defined as in \cref{c:f}. Let $G$ be the nonempty open subset of $\Gr(r,R'_d)$ satisfying the conclusions of \cref{thm: narrow}. The following then hold: 
\begin{enumerate}
\item $V\in \Gr(r,R_d)$;
\item $U=\langle \rho(f)\colon f\in V\rangle$;
\item If $U\in G$, then $\Ann_{R}(U^\perp+S_{\ge -d+1})=(V)_{\sym_n}$. 
\end{enumerate}
\end{thm}

\begin{proof}
(1) Since $U\in \Gr(r,R'_d)$, the vectors $\bsa^1, \dots,\bsa^r$ are linearly independent. It follows then from the definition of $f_{\bsa^i}$ that $f_{\bsa^1}, \dots, f_{\bsa^r}$ are linearly independent, and hence $\dim_\kk V=r$. 

(2) Observe that $\rho(f_{\bsa^i})=\sum_{\l\vdash d} \bsa^i_\l M_\l$ by \eqref{eq: alpha}. 

(3) By \cref{p:from W to f}, the ideals $(V)_{\sym_n}$ and $\Ann_{R}(U^\perp+S_{\geqslant -d+1})$ are equal in degree $d$. By \cref{thm: narrow}(3), the second ideal is also generated in degree $d$ when $U\in G$, and thus the ideals must be equal.
\end{proof}

\section{Two parametrizations for general $(r,d)$-symmetric ideals}\label{s: diagram}

Suppose $\kk$ is  infinite  with ${\rm char}(\kk)=0$ or $\ch(\kk)>n$. Fix positive integers $d$ and $r$. The goal of this section is to precisely define the notion of general $(r,d)$-symmetric ideals and deduce their main properties in \cref{thm: main}.

Towards this end, recall the map $\Phi: \Gr(r,R_d) \to \{$symmetric $(r,d)$-ideals$\}$ introduced in \eqref{Grassmannian intro}, which parametrizes our ideals of interest in terms of points in a Grassmannian variety. Then we consider a specific nonempty Zariski open set of this Grassmannian such that the ideals parametrized by points in that open set are our general $(r,d)$-symmetric ideals. 

Recall that the above-mentioned Grassmannian embeds into $\P(\bigwedge ^r (\kk^N))$ via the Pl\"ucker map. To describe this embedding consider the rational map
\begin{equation}\label{eq: matrix to Grassmannian}
\pi: \P:=\prod_{i=1}^r \P(R_d) \dashrightarrow \P \left(\bigwedge^r R_d \right)
\end{equation}
given by $\pi(v_1, \ldots, v_r)=v_1\wedge \cdots \wedge v_r$. This  map is defined on the open set $\mathcal{U}$  where the vectors $v_1,\ldots, v_r$ are linearly independent and its image is $\Gr(r,R_d)$.

\begin{lem}\label{lem: previous}
The following  is an open subset of $\Gr(r,R_d)$
\[
\widetilde O:=\left\{V\in \Gr(r,R_d)\colon \dim_\kk[(V)_{\sym_n}]_d\ge \dim_\kk R_d-\max\{P(d)-r,0\}\right\}.
\]
If $n$ is sufficiently large so that \cref{c:f} can be done, then $\widetilde{O}$ is nonempty. 
\end{lem}

\begin{proof}
    Let $A$ be an $N\times r$ matrix associated to $V\in \widetilde O$, where $N=\dim_\kk R_d$. More precisely, if $A=(a_{i,j})_{1\le j\le N,1\le i\le r}$, then $f_j=\sum_i a_{i,j}\mu_i$ describes a polynomial in $R_d$, where $\mu_1, \dots, \mu_N$ denote the monomial generators of $R_d$, listed in a fixed order, such that $f_1, \dots, f_r$ is a basis of $V$. Then consider the $N\times (r\cdot n!)$ matrix $A'$ obtained from the coefficients of the polynomials $\sigma\cdot f_i$ with $\sigma\in \sym_n$ and $1\le i\le r$ with respect to the chosen monomial basis. Set $M=\dim_\kk R_d-\max\{P(d)-r,0\}$. Since $\widetilde O$ is nonempty it follows that $M \leq r\cdot n!$ as there is at least one such $N\times (r\cdot n!)$ matrix of rank $M$. The set of those $V$ satisfying $\dim_\kk[(V)_{\sym_n}]_d<M$ is described by the vanishing of the $M\times M$ minors of $A'$. Since the minors are homogeneous polynomials in the coefficients of the matrix $A$, this set is closed in the projective space $\P$ of \eqref{eq: matrix to Grassmannian}, and thus its complement $O'$ is open. 
 
    It remains to show that the image of $O'$ under the map $\pi$ \eqref{eq: matrix to Grassmannian} is a Zariski open subset of the Grassmannian. Indeed, since the Grassmannian $\Gr(r,R_d)$ is a geometric quotient of $\P$, for any subset $O'$ of $\P$ the image $\widetilde{O}=\pi(O')\subseteq \Gr(r,R_d)$ is open if and only if $O'$ is open.  (This can also be seen directly by noticing that if we invert a maximal minor of an $r\times N$ matrix $A$, then one may assume that this minor is the identity matrix, hence locally the map $\pi$ is a projection, which is a continuous, open map.) 
    
    Assume now $n$ is sufficiently large, so that the conclusions of \cref{thm: gens} hold. When $r<P(d)$, the fact that $\widetilde{O}$ is nonempty is a consequence of \cref{thm: gens}(3) and \cref{thm: narrow}.  If $r\ge P(d)$, then one can take $V$ so that $V\supseteq \Span_\kk(x^\l\colon \l\vdash d)$, and see that $[(V)_{\sym_n}]_d=R_d$,  hence $V\in \widetilde O$. 
\end{proof}

\begin{rem}
\label{rem:Noah}
 An effective version for the last statement of \Cref{lem: previous} can be deduced from \cite[Proposition 2.10 and Theorem 4.8]{Walker}. It follows from the cited results that when $n>d$ the set $\widetilde{O}$ is nonempty when 
 \begin{equation}
 \label{e:n-Noah}
     n\ge  1+\frac{1}{r} \sum_{i=0}^{d-1} P(i).
\end{equation}
Furthermore, when $r\le P(d)$, the set $\widetilde{O}$ is nonempty if and only if \eqref{e:n-Noah} holds. In particular, \cref{lem: previous} shows that if $n$ is sufficiently large so that \cref{c:f} can be done, then \eqref{e:n-Noah} must hold.
\end{rem}

Recall from \eqref{Grassmannian intro} the map
 \begin{equation}\label{alpha map new}
\alpha: \Gr(r,R_d)\to \bigcup_{0\leq i\leq r} \Gr(i,R'_d)
\qquad \alpha(V):=\Span_{\kk}(\rho(f) : f\in V). 
\end{equation}

 \begin{lem}
\label{lem: alpha-continuous}
Assume $r<P(d)$. If $G$ is an open subset of $\Gr(r,R'_d)$ then $\alpha^{-1}(G)$ is a nonempty open subset of  $\Gr(r,R_d)$. 
\end{lem}

\begin{proof}
Let $\lambda_1, \dots \lambda_{P(d)}$ be an enumeration of the partitions of $d$,  $N=\dim_\kk R_d$, and order the monomials of $R_d$ as $\mu_1, \dots, \mu_N$, where $\mu_1=x^{\lambda_1}, \dots, \mu_{P(d)}=x^{\lambda_{P(d)}}$. 
As in the proof of \Cref{lem: previous}, let  $A$ be an  $N\times r$  matrix associated to $V\in \Gr(r, R_d)$, where $N=\dim_\kk R_d$. Namely, if $A=(a_{i,j})_{1\le j\le N,1\le i\le r}$, then $f_1, \dots, f_r$ is a basis of $V$, where $f_j=\sum_i a_{i,j}\mu_i$.

The action of an element $\sigma\in \sym_n$ on $V$ can be described as  multiplication of $A$ by an $N\times N$ invertible matrix $C_\sigma$. Set $C=\frac{1}{n!} \left(\sum_{\sigma\in \sym_n} C_\sigma \right)$. Let $B$ be the matrix consisting of the first $P_n(d)$ rows of the product $CA$. If $B=(b_{i,j})_{1\le j\le P(d),1\le i\le r}$, then the forms $F_1, \dots, F_r$ span $\alpha(V)$, where $F_j=\sum_i b_{i,j}M_{\lambda_i}\in R'_d$ and .

Applying the Cauchy-Binet formula it follows that for any subset $I$ of rows of $B$ of size $r$, and with $J$ equal to the set of all columns of $B$, the corresponding $r\times r$ minor  can be expressed as
\begin{equation}\label{eq: Cauchy-Binet}
\det(B)_{IJ}=\sum_{L}\det C_{IL}\det A_{LJ},
\end{equation}
where the sum ranges over the $\binom{N}{r}$ subsets $L$ of size $r$ of the columns of $C$ or of the rows of $A$. 

The set $\{V \mid \dim_\kk\alpha(V)=r\}$ is the complement of the vanishing locus of the minors $\det(B)_{IJ}$. The latter locus is cut out by linear equations in the Pl\"ucker coordinates of $V$ by \eqref{eq: Cauchy-Binet}.  Thus, the set of all $V\in \Gr(r,R_d)$ for which $\alpha(V)\in \Gr(r,R'_d)$ is open, and the restriction of $\alpha$ to this open set is a continuous map which is linear in each component cf.~\eqref{eq: Cauchy-Binet}. These two statements give the conclusion that $\alpha^{-1}(G)$ is open. To see that $\alpha^{-1}(G)$ is nonempty, recall that $G$ is nonempty and observe that if $V\in G$ is viewed as a subspace of $R_d$ then $\alpha(V)=V$. 
\end{proof}

\begin{notation}
\label{not:O}
When $n>d$  and $r<P(d)$ we let $G$ be the nonempty open subset of $\Gr(r,R'_d)$ from \cref{thm: narrow}. Then $\alpha^{-1}(G)$ is a nonempty open set of $\Gr(r, R_d)$ by \cref{lem: alpha-continuous}.  

We set 
\[
O=\begin{cases}
    \widetilde O\cap \alpha^{-1}(G) &\text{if $r<P(d)$}\\
    \widetilde O &\text{if $r\ge P(d)$.}
    \end{cases}
\]
This is an open set in $\Gr(r,R'_d)$. When $n$ satisfies \eqref{e:n-Noah}, $O$ is nonempty according to \cref{rem:Noah}.
\end{notation}

\begin{defn} 
\label{d: def-gsi}
Assume $n>d$ and $n$ satisfies \eqref{e:n-Noah}. We say that an ideal $I$ of $R$ is a {\it general $(r,d)$-symmetric ideal} if $I=(V)_{\sym_n}$ for some $V\in O$. 
\end{defn}

\cref{d: def-gsi} provides a  parametrization of general $(r,d)$-symmetric ideals by a nonempty open  set of the Grassmannian $\Gr(r,R_d)$ which  justifies the use of the word ``general''. This parametrization was introduced in \eqref{Grassmannian intro} as the map $\Phi$ and is recalled below in \eqref{Grassmannian}. Our next goal is to establish the equivalence of the  parametrization of general $(r,d)$-symmetric ideals by $\Gr(r,R_d)$ via $\Phi$ to another one by  $\Gr(r,R'_d)$ via the map $\Psi$ in \eqref{first map psi} and recalled below in \eqref{psi map}. The one given by the map $\Psi$ is better, since, as proven in \cref{t: CD}, $\Psi|_G$ is bijective.

\begin{notation}
\label{n:maps}We let $\mathcal{SI}_{r,d}(R)$ denote the set of all $(r,d)$-symmetric ideals of $R$ and then define: 
\begin{align}
&\overline{\alpha}: \mathcal{SI}_{r,d}(R) \to  \bigcup_{0\leq i\leq r} \Gr(i,R'_d)
\qquad &&\overline{\alpha}(I):=\alpha(I_d) \label{alpha-bar map}\\
&\Phi: \Gr(r,R_d) \to \mathcal{SI}_{r,d}(R)\qquad &&\Phi(V)\coloneq (V)_{\sym_n} \label{Grassmannian}\\
&\Psi\colon \Gr(r,R'_d)\to \mathcal{SI}_{r,d}(R)\qquad &&\Psi(U)\coloneq \Ann_R(U^\perp+S_{\geqslant -d+1}). \label{psi map}
\end{align}
\end{notation}

\begin{lem}\label{lem: phi subset psi}
Assume $n>d$, $r<P(d)$ and $n$ satisfies \eqref{e:n-Noah}. Then for a general $(r,d)$-symmetric ideal $I$ there exists $U\in G$ such that $I=\Psi(U)$. 
\end{lem}

\begin{proof}
 By the definition of an general $(r,d)$-symmetric ideal, $I=(V)_{\sym_n}$ with $V\in O$. Set $U=\alpha(V)$. Then $U\in G$.

Set $J=\Psi(U)$. 
\Cref{lem: perp map} shows that $V\subseteq J$, and since  $J$ is symmetric, we have
\begin{equation}
\label{e: Phi-sub-Psi}
\Phi(V)=(V)_{\sym_n}\subseteq J.
\end{equation}
Since $V\in\widetilde O$, we have the inequality below
\[
\dim_\kk I_d\ge \dim_\kk R_d-(P(d)-r)=\dim_\kk J_d
\]
where the equality comes from \cref{thm: narrow}. Thus the inclusion in \eqref{e: Phi-sub-Psi} implies $I_d=J_d$. Further, \cref{thm: narrow} gives that $J$ is generated in degree $d$ as is $I$ by definition. Thus we conclude that $I=J$. 
\end{proof}

\cref{lem: phi subset psi} shows that general symmetric ideals are in $\Psi(G)$, that is, $\Phi(O)\subseteq \Psi(G)$. Using \cref{thm: gens} to lift points in $G\subset\Gr(r, R'_d)$ to points in $O\subset \Gr(r, R_d)$ explicitly we establish the equality of these sets under appropriate conditions in \cref{t: CD}.

\begin{rem}
\label{rem: reformulate}
    In terms of \cref{n:maps} we can reformulate \cref{thm: gens} in a manner which reveals that $\overline{\alpha}$ and $\Psi$ are mutually inverse when restricted to appropriate subsets of their domains. With the assumptions of \cref{thm: gens}, we have
    \begin{enumerate}
    \item $\alpha(V)=\overline{\alpha}(\Phi(V))=U$.
    \item If $U\in G$, then $\Psi(U)=\Phi(V)$. In particular, $\overline{\alpha}(\Psi(U))=U$.
    \end{enumerate}
\end{rem}

In what follows, we denote by  $\Phi|_O$, $\Psi|_G$, $\alpha|_O$, $\overline{\alpha}|_{\Psi(G)}$ the restrictions of the maps $\Phi$, $\Psi$, $\alpha$ respectively $\overline{\alpha}$ to the indicated sets. We also make the convention to restrict the codomains of these maps as indicated in the diagram.

Recall that diagram \eqref{eq: CD} from the introduction aims to reconcile the two previously mentioned parametrizations $\Phi$ and $\Psi$ of $(r,d)$-symmetric ideals. Below we show that this diagram commutes. In view of \eqref{psi map} it suffices to show that the upper left triangle of \eqref{eq: CD} commutes. This appears as \eqref{eq: CDnew} below.

\begin{thm}
\label{t: CD}  
Suppose $\kk$ is  infinite  with ${\rm char}(\kk)=0$ or $\ch(\kk)>n$ and let $d,r$ be positive integers with $r\le P(d)$. 
Assume $n$ is sufficiently large so that \cref{c:f} can be done. Then $\Psi(G)=\Phi(O)$ and the map $\Phi$ restricted to $O$ admits the factorization  presented in the commutative diagram below: 
\begin{equation}\label{eq: CDnew}
\begin{tikzcd} [row sep =8mm, column sep = 3.2mm]
    O \arrow[rightarrow]{rrrrrr}{\Phi|_O} \arrow[rightarrow, drrr, "\alpha|_O"'] & & & &&& \Phi(O)\!\arrow[r,symbol={=}]&\!\Psi(G)  \\
   &   & & G \arrow[urrr,  "\Psi|_G"'] & & & &
\end{tikzcd}
\end{equation}
where   $\alpha|_O$ is surjective and $\Psi|_G$ is bijective, with $(\Psi|_G)^{-1}=\overline{\alpha}|_{\Psi(G)}$. Furthermore, $O=\Phi^{-1}(\Psi(G))$. 
\end{thm}

\begin{proof} First we show that if $U\in G$, then $\Psi(U)\in \Phi(O)$. Indeed, if $U\in G$, then \cref{rem: reformulate} gives that $\Psi(U)=\Phi(V)$ and $\alpha(V)=U$ for $V$ as in \cref{thm: gens}.  Furthermore, if $J=\Psi(U)$, then $\dim_\kk J_d=\dim_\kk R_d-(P(d)-r)$ by \cref{thm: narrow}. Consequently, $V\in O$ and $\Psi(U)\in \Phi(O)$. This shows $\Psi(G)\subseteq \Phi(O)$. 

Now, starting with $V\in O$, set $U=\alpha(V)\in G$. By \cref{lem: phi subset psi} we conclude that $\Phi(V)=\Psi(U)$. This shows that \eqref{eq: CDnew} is commutative and also that $\Phi(O)\subseteq \Psi(G)$. Since we also proved above the reverse inclusion, we conclude $\Phi(O)=\Psi(G)$. 

The map $\Psi|_G$ is surjective, given the restriction of the codomain to $\Psi(G)$. Since $\Psi$ is also injective, it must be bijective. Since $\Phi|_O$ is surjective, then the commutativity of the diagram implies $\alpha|_O$ is surjective.

Since $\Phi(O)=\Psi(G)$, we have $O\subseteq \Phi^{-1}(\Psi(G))$. To prove the reverse inclusion, let $V\in \Phi^{-1}(\Psi(G))$. Then $\Phi(V)\in \Psi(G)$, and hence \cref{thm: narrow} shows that $\dim_\kk \Phi(V)=\dim_\kk R_d-(P(d)-r)$ and thus $V\in \widetilde O$. Write $\Phi(V)=\Psi(U)$ with $U\in G$. By \cref{rem: reformulate} we have that $\overline{\alpha}(\Psi(U))=U$, and hence $\overline{\alpha}(\Phi(V))=U\in G$. Finally, observe that $\alpha(V)=\overline{\alpha}(\Phi(V))$, and we conclude that $\alpha(V)\in G$, hence $V\in \alpha^{-1}(G)\cap \widetilde O=O$. We have showed thus the inclusion $\Phi^{-1}(\Psi(G))\subseteq O$, and hence $\Phi^{-1}(\Psi(G))= O$. 

The equality $\overline{\alpha}(\Psi(U))=U$ for all $U\in G
$ and the fact that $\Psi|_G$ is bijective also shows that $\Psi|_G$ and $\overline{\alpha}|_{\Psi(G)}$ are inverses. 
\end{proof}

We are  ready to summarize the properties of general $(r,d)$-symmetric ideals by putting together \cref{thm: narrow} and results from \cite{HSS24}. 

\begin{thm}\label{thm: main}
Suppose $\kk$ is  infinite  with ${\rm char}(\kk)=0$ or $\ch(\kk)>n$ and fix  positive integers $r$, $d$. Assume  
\[
n\ge  \max\left\{d+1, 1+\frac{1}{r} \sum_{i=0}^{d-1} P(i)\right\}
\]
Then a general $(r,d)$-symmetric ideal $I$ of ${\kk}[x_1,\ldots, x_n]$ satisfies the following: 

\begin{enumerate}
\item The Hilbert function of $A=R/I$ is given by
\[
 \HF_A(i):=\dim_\kk A_i = \begin{cases}\dim_\kk R_i &\text{if $i\le d-1$}\\
\max\{P(d)-r,0\}&\text{if $i=d$}\\
 0 &\text{if $i>d$.}\
 \end{cases}
\]
\item The betti table of $A$ has the form
\[
\begin{matrix}
     &0&1&2&\cdots&i &\cdots & n-2 &n-1&n\\
     \hline
     \text{total:}&1&u_1&u_2&\cdots&u_i&\cdots&u_{n-2} &u_{n-1}+\ell&a+b\\
     \hline
     \text{0:}&1&\text{.}&\text{.}&\text{.}&\text{.}&\text{.}&\text{.}&\text{.}&\text{.}\\
     \text{\vdots}&\text{.}&\text{.}&\text{.}&\text{.}&\text{.}&\text{.}&\text{.}&\text{.} &\text{.}\\
     \text{d-1:}&\text{.}&u_1&u_2&\cdots&u_i&\cdots&u_{n-2} &u_{n-1}&b\\
     \text{d:}&\text{.}&\text{.}&\text{.}&\text{.}&\text{.}&\text{.}&\text{.}&\ell & a \\
     \end{matrix}
\]
with 
\begin{eqnarray*}
\ell &=& \max\{P(d)-P(d-1)-r,0\} \\
a &=&\max\{P(d)-r,0\},\\
 b&=&
 \binom{n+d-2}{d-1}-an+ \ell  \\
  u_{i+1} &=& \binom{n+d-1}{d+i}\binom{d+i-1}{i}-a\binom{n}{i}  
  \end{eqnarray*}
\item $A$ has the Weak Lefschetz Property. 
\end{enumerate}

Moreover,  the Poincar\'e series of all finitely generated graded $A$-modules are rational, sharing a common denominator. When $d>2$ or $d=1$ the ring $A$ is Golod. When $d=2$, then $A$ is Koszul. The ring $A$ is Gorenstein if and only if $d=2$ and $r=1$, or $d=1$. 
\end{thm}

\begin{proof}
Assume first $r\ge P(d)$. Then $I=(V)_{\sym_n}$ with $V\in \widetilde O$, and hence $I=\m^d$ follows from the definition of $\widetilde O$. In this case, the betti numbers of $I$ are well-known and agree with the ones in (2), and the Weak Lefschetz Property holds trivially. In this case, $A$ is known to be Golod, it is Koszul when $d=2$, and it is trivially Gorenstein when $d=1$.

Assume $r<P(d)$, which implies $d\geq 2$.  By \cref{lem: phi subset psi} we have  $I=\Ann_R(U^\perp+S_{\geqslant -d+1})$ for some $U\in G$. 
Then \cref{thm: narrow} implies that $A$ is $d$-extremely narrow, with socle polynomial equal to 
 \[
 \big(\dim_\kk R_{d-1}-(P(d)-r)n+ \max\{P(d)-P(d-1)-r,0\}\big)z^{d-1}+(P(d)-r)z^d\,.
 \]
 and $A$ has the Weak Lefschetz Property.
 The Hilbert series of $A$ follows then from \cite[Lemma 4.5]{HSS24} and the betti numbers of $A$ follow from \cite[Corollary 4.10]{HSS24}. The considerations about Poincar\'e series follow from \cite[Proposition 4.13]{HSS24}. To assess the Gorenstein property we set $a=1$ and $b=0$ and observe that the only possible values for $r,d$ are as claimed.
 \end{proof}

A consequence of \cref{thm: main} is that general $(r,d)$-symmetric ideals have minimal Hilbert function among all $(r,d)$-symmetric ideals. 
 \begin{cor}\label{cor: extremal HF}
 Assume the setting of \cref{thm: main}. If $I$ is a general $(r,d)$-symmetric ideal and $J$ is any $(r,d)$-symmetric ideal, then 
\[
\HF_{R/I}(i)\le \HF_{R/J}(i) \qquad\text{for all $i$}
\]
 \end{cor}
 \begin{proof}
 Since both $I$ and $J$ are generated in degree $d$ and $[R/I]_{d+1}=0$, we only need to prove the inequality for $i=d$.  The inequality is clear when $I=\m^d$, so we assume $r<P(d)$. 

Consider the map $\rho\colon R\to R'$, and let $J'=\rho(J)$. Then we have an induced surjective map $R/J\to R'/J'$, giving
\[
\HF_{R/J}(d)\ge \HF_{R'/J'}(d)=P(d)-\dim_\kk([J']_d)\ge P(d)-r=\HF_{R/I}(d)
\]
where the second inequality uses the inequality $\dim_\kk([J']_d)\le r$, which is justified by the fact that $J'$ is generated by $r$ elements in degree $d$. 
 \end{proof}

 We end this section with a discussion of a possible generalization of our results. 

 \begin{defn}
 \label{d: multiple}
Let $s\ge 1$. For each $1\le i\le s$ consider integers  $r_i\ge 1$ and $d_i$ such that $2\le d_1< d_2< \dots<d_s$.  We say that an ideal $I$ is an $(r_1, \dots, r_s; d_1, \dots, d_s)$--{\it symmetric ideal} if  $I=(V)_{\sym_n}$, where 
\[
V=V_{d_1}\oplus+\dots \oplus V_{d_s}
\]
with $V_{d_i}\in \Gr(r_i, R_{d_i})$. 
 \end{defn}

With the notation in \cref{d: multiple}, one may wonder if we can extend our results using a similar notion of general $(r_1, \dots, r_s;d_1, \dots, d_s)$--symmetric ideal, that would involve a nonempty open set in $\Gr(r_1, R_{d_1})\times \dots \times \Gr(r_s,R_{d_s})$. However, it turns out that such an extension does not produce any additional information, as pointed out below, since a general  $(r_1, \dots, r_s; d_1, \dots, d_s)$-symmetric ideal would be equal to a general $(r_1,d_1)$--symmetric ideal. 

 \begin{rem}\label{rem: multiple}
With the notation in \cref{d: multiple}, if $V_{d_1}\in O$, where $O$ is the open set of \cref{not:O} with $r=r_1$ and $d=d_1$, then
 \[
 I=(V_{d_1})_{\sym_n}\,.
 \]
 Indeed, the assumption that $V_{d_1}\in O$ gives that $(V_{d_1})_{\sym_n}$ is a general $(r_1,d_1)$--symmetric ideal, and thus $\HF_{R/(V_{d_1})_{\sym_n}}(i)=0$ for $i>d_1$. This implies that 
 \[
 I_i\supseteq [(V_{d_1})_{\sym_n}]_i=\m^i \quad\text{for all $i>d_1$}
 \]
 and hence the inclusion must be an inequality. Since we also have $I_i=[(V_{d_1})_{\sym_n}]_i$  when $i=d_1$ for degree reasons, the two ideals must be equal. 
\end{rem}

\section{Symmetric inverse systems}\label{s: inverse system}

In this section we further develop the theme that most $(r,d)$-symmetric ideals $I$ can be understood by examining their $\sym_n$-invariant elements, that is, by considering the ideal $I'=I\cap R'$, where $R'$ is the ring of invariants of $R$ under the action of $\sym_n$. The main result of this section, \cref{thm: I' inverse}, shows that the ideal $I'$, although merely a shadow of $I$, can be used to recover $I$ by means of  Macaulay-Matlis duality.

In what follows, if $W$ is a subset of $S$ then we denote by $R\circ W$ the $R$-submodule of $S$ generated by $W$. We observe that the dual pairing $\circ$ between $R$ and $S$ induces a dual pairing between $R'$ and $S'$, that we also denote $\circ$. If $I'$ is an ideal of $R'$, then we denote $(I')^{-1}$ the $R'$-submodule of $S'$ given by 
\[(I')^{-1}=\{g\in S'\colon f\circ g=0 \text{ for all } f\in I'\}\,.
\]

If $I$ is a general $(r,d)$-symmetric ideal we show that the following diagram commutes 
\begin{equation}\label{eq: CD algebra}
\begin{tikzcd}
    I \arrow[]{rr}{} \arrow[]{d}{\rho} & &I^{-1}=W+S_{\geq -d+1} \arrow[]{d}{\rho} \arrow {ll}\\[5pt]
     I'\arrow{rr} &&(I')^{-1}=W+S'_{\geq -d+1} \arrow {ll}.
\end{tikzcd}
\end{equation}
The vertical maps $\rho$ are given by the Reynolds operator $\rho(f)=\frac{1}{n!} \sum_{\sigma\in \sym_n} \sigma \cdot f$ while the horizontal maps arise from Macaulay-Matlis duality with respect to  the perfect pairings between $R$ and $S$ (top) and $R'$ and $S'$ (bottom) respectively. The surprising aspect of the diagram is that $I^{-1}$  is obtained from $(I')^{-1}$ by extension of scalars, that is, $I^{-1}=R\circ ({I'})^{-1}$, whereas it is not the case that $I$  is obtained from $I'$  by extension of scalars.

Next we  establish the circle of ideas pictured in diagram \eqref{eq: CD algebra}.

\begin{thm}\label{thm: I' inverse}
Suppose $\kk$ is  infinite  with ${\rm char}(\kk)=0$ or $\ch(\kk)>n$ and let $d$, $r$ be positive integers.
Assume $n$ is sufficiently large so that \cref{c:f} can be done and $r\le P(d)$.
If $I$ is a general $(r,d)$-symmetric ideal in $R$, $I'=I\cap R'$ and $W=(\overline{\alpha}(I))^\perp$, then the following identities hold 
 \begin{eqnarray*}
 (I')^{-1} &=& W+S'_{\geqslant -d+1}\\
 I^{-1} &=& W+S_{\geqslant -d+1}=R\circ (W+S'_{\geqslant -d+1})=R\circ (I')^{-1}.
 \end{eqnarray*}
 Thus $I'$ determines $I$ by means of the identity $I=\Ann_R\left(R\circ (I')^{-1}\right)$.
\end{thm}
\begin{proof}
Since $I$ is general, \cref{t: CD} gives $I=\Psi(U)$ with $U\in G$, where $G$ is the Zariski open set from \cref{thm: narrow} and $U=\overline{\alpha}(I)$. We have thus  $W=U^\perp$ and hence
\[
I=\Psi(U)=\Ann_R(W+S_{\geqslant -d+1})\,.
\]
By duality, it follows that
\begin{equation}\label{eq: I^-1}
I^{-1}=W+S_{\geqslant -d+1}.
\end{equation}

Next we prove that the vector space above is obtained from $W+S'_{\geq -d+1}$ by extension of scalars from $R'$ to $R$. First notice that $R\circ S'=S$. Indeed, it suffices to prove that any monomial in $S$ is obtained from a symmetric polynomial via contraction. Let $\mu=y_1^{b_1}\cdots y_n^{b_n}$ be such a monomial and let $q=\max\{b_i: 1\leq i\leq n\}$. Then $y_1^{q}\cdots y_n^{q}\in S'$ and $x_1^{q-b_1}\cdots x_n^{q-b_n}\circ y_1^{q}\cdots y_n^{q}=\mu$. Next, observe that $R\circ W\subseteq W+S_{\geqslant -d+1}$, since $W\subseteq S_{-d}$. Therefore we have established containments
\[
 W+S_{\geq -d+1} \subseteq R\circ (W+S'_{\geq -d+1})=R\circ W+R\circ S'_{\geqslant -d+1}\subseteq W+S_{\geqslant -d+1},
\]
which yield the desired equality
\begin{equation}\label{eq: extension}
W+S_{\geqslant -d+1}=R\circ (W+S'_{\geqslant -d+1}).
\end{equation}

Finally we show  $(I')^{-1}=W+S'_{\geqslant -d+1}$. By definition,  $I$ is generated in degree $d$, and, since $I'=\rho(I)$, the ideal $I'$ is also generated in degree $d$.
 Moreover, $U=[I']_d$ and $W=U^\perp= [(I')^{-1}]_{-d}$.  In view of the containment  $S'_{\geqslant -d+1}\subseteq (I')^{-1}$, which follows from $I'$ being generated in degree $d$, we have established $W+S'_{\geqslant -d+1}\subseteq (I')^{-1}$. 
 
 On the other hand, we have
\begin{equation}
\label{e:ann}
 \Ann_{R'}(W+S'_{\geqslant-d+1})=\Ann_{R'}(W)\cap \Ann_{R'}(S'_{\geqslant -d+1})=\Ann_{R'}(W)\cap R'_{\geqslant d}.
 \end{equation}
 Note that 
 \begin{align}
 \begin{split}
 \label{e:annincl}
 \Ann_{R'}(W)&=\{f\in R'\mid f\circ g=0 \text{ for all $g\in W$}\}\\
 &\subseteq \{f\in R\mid f\circ g=0 \text{ for all $g\in W$}\}\\
 &=\Ann_R(W).
 \end{split}
 \end{align}
Putting \eqref{e:ann} and \eqref{e:annincl} together, we conclude that
\[
\Ann_{R'}(W+S'_{\geqslant-d+1})\subseteq \Ann_R(W)\cap R_{\geqslant d}=\Ann_R(W+S_{\geqslant -d+1})=I.
\]
Upon intersecting with $R'$, we obtain
\[
\Ann_{R'}(W+S'_{\geqslant-d+1})\subseteq I'.
\]
Equivalently, $(I')^{-1}\subseteq W+S'_{\geqslant-d+1}$. We have proved thus  $(I')^{-1}=W+S'_{\geqslant-d+1}$. 
\end{proof}

\section{Asymptotic stability with respect to the number of variables}\label{s: asymptotic}

In this section, we are concerned with how the properties of interest of general symmetric ideals change with the number of variables. 

We denote  $R=\kk[x_1, \dots, x_n, \ldots]$, with $\kk$ a field, the polynomial ring in countably many variables and $R_n=\kk[x_1, \dots, x_n]$. Unlike in previous sections, the subscripts utilized so far record the Krull dimension of each ring, not internal degrees. The ring $R$ is the colimit of the directed system of rings $R_n$.

Let $\sym_\infty=\bigcup_{n\geq 1} \sym_n$ be the group of permutations of $\N$ that fix all but finitely many elements. It acts on $R$ by permuting variables in the natural way, that is,
\[
\sigma\cdot f = f(x_{\sigma(1)}, x_{\sigma(2)}, \ldots ) \text{ for any } \sigma \in \sym_\infty.
\]
This action induces an action of $\sym_n$ on $R_n$ for every $n \in \N$ that coincides with that considered in the previous sections.

An ideal $\I$ of $R$ is {\em $\sym_\infty$-invariant} if $\sigma\cdot f\in \I$ whenever $f\in \I$.  If $\I$ is a $\sym_\infty$-invariant ideal, then the sequence of ideals $I_n = \I \cap R_n$ are symmetric with respect to the action of $\sym_n$ on $R_n$. Moreover the chain satisfies the property that for $m\leq n$, $\sigma\in\sym_n$ and $f\in I_m$, it follows that $\sigma \cdot f\in I_n$. A sequence $\{I_n\}_{n\in\N}$ of ideals that has this property is called a {\em $\sym_\infty$-invariant chain}. Conversely, given a $\sym_\infty$-invariant chain $\{I_n\}_{n\in\N}$, the colimit  $\I=\varinjlim_{n\in\N} I_n$ is a $\sym_\infty$-invariant ideal. 

As $R$ is Noetherian up to symmetry by \cite{Cohen}, any $\sym_\infty$-invariant ideal can be generated by finitely many elements $f_1,\ldots, f_r$ of $R$, thus has the form
\[
\I=(f_1, \ldots, f_r)_{\sym_\infty}=\left (\sigma \cdot f_i \mid \sigma\in \sym_\infty, 1\leq i\leq r \right ).
\]
Setting $m$ to be the largest index of a variable in the union of the supports of $f_1, \ldots, f_r$ gives that $\I$ is the colimit of the chain $I_m\subseteq I_{m+1}\subseteq \cdots I_n \subseteq \cdots$, where $I_n=(f_1, \ldots, f_r)_{\sym_n}$ for each $n\geq m$.

We apply the results from our previous sections to study chains of  $\sym_\infty$-invariant ideals as above where $f_1, \ldots, f_r$ are general forms of degree $d$. Since $n$ is no longer fixed, an adjustment of our previously used notation is needed. 

\begin{notation} For each $n\ge d$, we identify the $\kk$-vector spaces $[R_n']_d$ with $ \kk^{P(d)}$ by identifying a basis  of $[R_n']_d$ with the standard basis of $\kk^{P(d)}$. This leads to an identification of $\Gr(r, [R_n']_d)$ with $\Gr(r, \kk^{P(d)})$ for all $n\ge d$. In \cref{thm: narrow}, we established existence of a nonempty open subset $G$ of the Grassmannian $\Gr(r, [R_n']_d)$ with certain properties, and we made the point that this set is independent of $n$ when $n\ge d$, in the sense that the defining equations (in Pl\"ucker coordinates) of its complement have coefficients that do not depend on $n$. Via the identifications discussed above, we may think of this set as a nonempty open subset $G$ of  $\Gr(r, \kk^{P(d)})$ whenever $n\ge d$.  While this set $G$ is independent of $n$, the set $O$ and the maps $\alpha$, $\Phi$, $\Psi$  defined in  \eqref{alpha map new}, \eqref{Grassmannian}, and \eqref{psi map} depend on $n$, and in this section we denote them $O_n$, $\alpha_n$, $\Phi_n$ and $\Psi_n$ to indicate this dependence.  

Recall that $\mathcal{SI}_{r,d}(R_n)$ denotes the set of symmetric $(r,d)$ ideals in the ring $R_n$. To simplify notation, since $r$ and $d$ are fixed, we write $\mathcal{SI}(R_n)$ instead of $\mathcal{SI}_{r,d}(R_n)$.  We further let $\mathcal{SI}^g(R_n)$ denote the set of general $(r,d)$-symmetric ideals in $R_n$. 

Let $n,m$ be integers with $n\le m$. We define canonical maps 
\begin{align*}
\pi_{n,m}&\colon \Gr(r, [R_n]_d)\to \Gr(r, [R_m]_d) \qquad && \pi_{n,m}(V)\coloneq V \\
\eta_{n,m}&\colon \mathcal{SI}(R_n)\to \mathcal{SI}(R_m) && \eta_{n,m}(I)\coloneq I_{\sym_m}=\{\sigma\cdot f\colon \sigma\in \sym_m, f\in I\} 
\end{align*}

In \cref{t: CD} we established a commutative diagram, in which the maps involved are restrictions of the maps $\alpha$, $\Psi$, $\Phi$. In \cref{thm: CD-n} below, which compares the commutative diagrams given in \eqref{eq: CDnew} for different values of $n$, the maps involved are also restrictions, but we decided to omit the restriction symbols from the diagram (i.e. we write $\alpha_n$ instead of $\alpha_n|_{O_n}$), in order to avoid notation overload. 
\end{notation}

In the following theorem we prove that our family of general $(r,d)$-symmetric ideals is asymptotically constant upon increasing the number of variables in a sense made precise by the bijection $\eta_{n,m}$ below.

\begin{thm}
\label{thm: CD-n}
Suppose $\kk$ is a field  with ${\rm char}(\kk)=0$ and let $d,r$ be positive integers with $r\le P(d)$. 
If $m$ is sufficiently large to allow for \cref{c:f}, then for all $n\ge m$ the following hold: 
\begin{enumerate}
    \item The map $\eta_{m,n}$ restricts to a bijection $\eta_{m,n}\colon \mathcal{SI}^g(R_m)\to \mathcal{SI}^g(R_n)$. 
    \item $\pi_{m,n}(O_m)\subseteq O_n$. 
    \item There is a commutative diagram, in which all arrows in the rightmost triangle are bijections. 
 \begin{equation}\label{eq: CD2}
\begin{tikzcd}[column sep=5mm, row sep=5mm] 
  O_n \arrow{rrrrrrrr}{\Phi_n}\arrow[ddddrrrr, "\alpha_n"']  & & & & &  & & &\mathcal{SI}^g(R_n)  \\
  &   & &  O_m\arrow[dddr,"\alpha_m"] \arrow[]{rr}{\Phi_m}\arrow[ulll,"\pi_{m,n}"'] &  &  \mathcal{SI}^g(R_m)\arrow[urrr,"\eta_{m,n}"] &  & & \\
  \\
  \\
    && & & G \arrow[uuur,"\Psi_m"]\arrow[uuuurrrr,"\Psi_n"'] &&& &
\end{tikzcd}
\end{equation}   
    \end{enumerate}
    In particular, if $I$ is a general $(r,d)$-symmetric ideal of $R_m$, then $I_{\sym_n}$ is a general $(r,d)$-symmetric ideal of $R_n$ for all $n\ge m$.
\end{thm}

\begin{proof}
Let $U\in G$ with  
\[
U=\Span_\kk\left ( \sum_{\l\vdash d}\bsa^1_\l M_\l, \ldots,  \sum_{\l\vdash d}\bsa^r_\l M_\l \right ).
\]
By \cref{thm: gens} we have 
\begin{align*}
\Psi_n(U)&=(f_{\bsa^1}, \dots, f_{\bsa^r})_{\sym_n}\qquad \text{and}\\
\Psi_m(U)&=(f_{\bsa^1}, \dots, f_{\bsa^r})_{\sym_m}
\end{align*}
where the polynomials $f_{\bsa^i}$ are defined in \cref{c:f} and, when $n$ is large, do not depend on $n$. 
We conclude that
\begin{equation}
\label{e: Psi}
\Psi_n(U)=\eta_{m,n}\Psi_m(U) \qquad\text{for all $U\in G$}
\end{equation}
Since $\Psi_m(G)=\mathcal{SI}^g(R_m)$ and $\Psi_n(G)=\mathcal{SI}^g(R_n)$ by \cref{t: CD}, we conclude that $\eta_{m,n}(\mathcal{SI}^g(R_m))=\mathcal{SI}^g(R_n)$. 
The equation \eqref{e: Psi} establishes thus the commutativity of the right triangle in \eqref{eq: CD2}.  Furthermore, since the maps $\Psi_m\colon G\to \mathcal {SI}^g(R_m)$ and $\Psi_n\colon G\to \mathcal {SI}^g(R_n)$ are bijective by \cref{t: CD}, we conclude that the map $\eta_{m,n}\colon \mathcal{SI}^g(R_m)\to \mathcal{SI}^g(R_n)$ is bijective as well.

 It can be easily checked from the definitions that 
\begin{equation}
\label{eq: commute}
\eta_{m,n}\Phi_m(V)=\Phi_n\pi_{m,n}(V)\qquad\text{and}\qquad 
\alpha_n\pi_{m,n}(V)=\alpha_m(V)
\end{equation}
for all $V\in \Gr(r, [R_m]_d)$. 
We want to now show $\pi_{m,n}(O_m)\subseteq O_n$. Recall that by \cref{t: CD} we have  $O_m=\Phi_m^{-1}(\Psi_m(G))$ and $O_n=\Phi_n^{-1}(\Psi_n(G))$.
Let $V\in O_m$, hence $\Phi_m(V)\in \Psi_m(G)$. We also have
\[
\Phi_n(\pi_{m,n}(V))=\eta_{m,n}\Phi_m(V)\in \eta_{m,n}(\Psi_m(G))=\Psi_n(G)
\]
and this shows $\pi_{m,n}(V)\in \Phi_n^{-1}(\Psi_n(G))=O_n$. We therefore proved $\pi_{m,n}(O_m)\subseteq O_n$. 

The commutativity of the left triangle and of the trapezoid on top then follows from \eqref{eq: commute}. The commutativity of the middle triangle and the outer triangle comes from \cref{t: CD}. 
\end{proof}

Let $\{I_n\}_{n\geq 0}$ be a $\sym_{\infty}$-invariant chain of homogeneous ideals with colimit $\I$. Its {\em equivariant Hilbert series} is defined in \cite{NagelRomer} as a bivariate generating series encoding all the Hilbert series of the quotient rings $R_n/I_n$ simultaneously
\begin{equation}\label{eq: equiv HS}
H_{\I}(s,t)=\sum_{n\geq 0}H_{R_n/I_n}(t)s^n=\sum_{n\geq 0} \sum_{u\geq 0} \dim_\kk [R_n/I_n]_u t^us^n.
\end{equation}

We summarize the key features that the numerical invariants of $\sym_{\infty}$-invariant chains $\I$ are known to possess:
    \begin{enumerate}[leftmargin=2em]
        \item The equivariant Hilbert series $H_{\I}(s,t)$ is a rational function of the form
        \[
        H_{\I}(s,t)=\frac{g(s,t)}{(1-t)^a\prod_{i=1}^b\left ((1-t)^{a_i} -sf_i(t)\right)}
        \]
        for some integers $a,b, a_i$ and some polynomials $g\in\Z[s,t]$, $f_i\in\Z[t]$; see \cite[Proposition 7.2]{NagelRomer}.
        \item The Krull dimension and the height of $I_n$ are eventually linear functions of $n$: by \cite[Teorem 5.14]{Nagel3} and \cite[Theorem 7.10]{Nagel1} there exist $A,B, A',B'\in\Z$ such that
        \[ \dim(R_n/I_n)= An+B \text{ and } {\rm ht}(I_n)= A'n+B' \text{ for all } n\gg0.   \]
\item The multiplicity grows eventually exponentially by \cite{Nagel}. More precisely, there are $M,L\in\Z$  and $Q\in\Q$ such that 
\[
\lim_{n\to\infty}\frac{e(R_n/I_n)}{M^nn^L}=Q.
\]
\item The regularity of $I_n$ is bounded by a linear function of $n$: by \cite[Corollaries 4.6, 4.7]{Nagel2} there exist $C,D\in\Z$ such that
        \begin{equation}\label{eq: reg}
        {\rm reg}(I_n)\leq  Cn+D \text{ for all } n\gg0.   
        \end{equation}
\item For any $i \geq 0$, the $i$-th column of the Betti table of $I_n$ has a stable shape whenever $n\gg0$. More precisely, there exist integers $j_0 <\cdots < j_t$ depending on $i$ and $\I$ such that for $n\gg0$
\[
\beta_{i,j}(I_n)\neq 0 \text{ if and only if } j \in \{j_0, \ldots, j_t\}.
\]
For more precise statements regarding stabilization of Betti numbers see \cite{NagelRomerFI}.
    \end{enumerate}

It is conjectured in  \cite[Conjecture 1.1]{Nagel2} that equality holds in \eqref{eq: reg} for appropriate values of $C,D$ and that a similar property is true for projective dimension \cite[Conjecture 1.1]{Nagel1}, namely that there exist $E,F\in\Z$ so that
 \begin{equation}\label{eq: pd}
        {\rm pd}(I_n)\leq  En+F \text{ for all } n\gg0.   
        \end{equation}
Both conjectures are known to hold true when the quotients $R_n/I_n$ are artinian and also for $\sym_{\infty}$-invariant chains of monomial ideals \cite{Satoshi1}, \cite{Raicu1}.

We may also define the $(i,j)$-{\it equivariant betti series of $\I$} by 
\begin{equation}
\label{equiv-beta}
\beta_{i,j}(\I,s)=\sum_{n\geq 0}\dim_\kk \Tor_i^{R_n}(R_n/I_n,\kk)_j s^n
\end{equation}
for all $i,j$, and the {\it equivariant bigraded Poincar\'e series} of $\I$ by
\begin{equation}
\label{equiv-Poincare}
P_\I(s,t,u)=\sum_{i\ge 0, j\ge 0}\beta_{i,j}(\I,s)t^iu^j\,.
\end{equation}
While it seems that this question has not been explicitly stated so far, it is natural to ask whether $\beta_{i,j}(\I,s)$ is a rational function of $s$, and, furthermore, if $P_\I(s,t,u)$ is a rational function in $s,t,u$.

In the following result we describe the features of $\sym_{\infty}$-invariant chains of general $(r,d)$-symmetric ideals. A surprising outcome is that for these chains the algebraic invariants exhibit slower growth than for arbitrary $\sym_\infty$--invariant chains. For example, while the multiplicity of arbitrary $\sym_{\infty}$-invariant chains grows exponentially, for general chains it grows polynomially. 

\begin{thm}\label{cor: sym-infty chain} Suppose $\kk$ is  a field  with ${\rm char}(\kk)=0$  and let $d,r$ be positive integers with $r\le P(d)$. 
Assume $m$ is sufficiently large to allow for \cref{c:f}. Let $I$ be a general $(r,d)$-symmetric ideal of $R_m$, and consider the $\sym_{\infty}$-invariant chain $\{I_n\}_{n\in \N}$ with $I_n=(I)_{\sym_n}$ and colimit $\mathcal I$. Using the notation 
\begin{eqnarray*}
\ell &=& \max\{P(d)-P(d-1)-r,0\} \\
a &=&P(d)-r,
\end{eqnarray*}
the following hold: 
\begin{enumerate}
 \item $\dim(R_n/I_n)=0$  and ${\rm ht}(I_n)=n$ for all $n\ge m$, i.e, the rings $R_n/I_n$ are artinian. 
 \item $\pd_{R_n}(R_n/I_n)=n$ for all $n\ge m$
 \item $\reg_{R_n}(R_n/I_n)=d$ for all $n\ge m$
 \item the $i$-th betti numbers of $R_n/I_n$ are polynomials in $n$ of degree $d+i$ for all $0\leq i\leq n-1$, while the $n$-th betti number of $R_n/I_n$ is a linear function of $n$.
 \item The equivariant  Hilbert series in its reduced form is 
\[
H_\I(s,t)=\frac{s\left[\sum_{j=0}^{d-1}(1-s)^jt^{d-1-j}\right ]+a(1-s)^{d-1}t^d}{(1-s)^d}.
\]
 \item  The  multiplicity of $R_n/I_n$ grows polynomially as a function of $n$ with
 \[
 \lim_{n\to \infty}\frac{e(R_n/I_n)}{n^{d-1}}=\frac{1}{(d-1)!}.
 \]
 \item The equivariant betti number $\beta_{i,j}(\I,s)$ is a rational function of $s$ for all $i,j$, and the equivariant bigraded Poincar\'e series $P_\I(s,t,u)$ is a rational function of $s,t,u$.  More precisely,
 \begin{equation*}
P_\I(s,t,u)\!=\!\frac{1}{1-s}\!+\!stu^d\left(\frac{1}{(1\!-\!s)(1\!-\!s\!-\!stu)^d}\!-\!\frac{a}{(1\!-\!s\!-\!stu)(1\!-\!stu)}\!+\!\frac{\ell+ au +\ell su}{1-stu}\right).
\end{equation*}
\end{enumerate}
\end{thm}

\begin{proof}
By \cref{thm: CD-n} it follows that for every $n\geq m$ the ideal $I_n=(I_m)_{\sym_n}$ is a general symmetric ideal. 
Statements (1)--(4) then follow from the formulas in \Cref{thm: main}(2).

 Based on \Cref{thm: main}(1), the equivariant  Hilbert series is computed as follows
 \begin{eqnarray*}
 H_\I(s,t) &=&\sum_{n\geq 1}\left[\sum_{i=0}^{d-1} \binom{n-1+i}{i}t^i +at^d\right]s^n \\
  &=&\sum_{i=0}^{d-1}\left[ \sum_{n\geq 1} \binom{n-1+i}{i} s^n \right] t^i + at^d\frac{1}{1-s}=\sum_{i=0}^{d-1}\frac{s}{(1-s)^{i+1}}  t^i +a\frac{t^d}{1-s}\\
  &=&\frac{s}{(1-s)}\frac{1-\frac{t^d}{(1-s)^d}}{1-\frac{t}{(1-s)}}+a\frac{t^d}{1-s}=\frac{s[(1-s)^d-t^d]}{(1-s-t)(1-s)^d}+a\frac{t^d}{1-s}\\
  &=&\frac{s\left[\sum_{j=0}^{d-1}(1-s)^jt^{d-1-j}\right ]+a(1-s)^{d-1}t^d}{(1-s)^d}.
 \end{eqnarray*}

 Using \Cref{thm: main}(1),  the  multiplicity of $R_n/I_n$ is given by
 \[
 e\left(\frac{R_n}{I_n}\right)=\binom{n+d-1}{d-1}+P(d)-r,
 \]
whence the limit claimed in (6) follows.

(7) Note that the formulas for Betti numbers of \cref{thm: main} apply both for the case $r\leq P(d)$ and $r\geq P(d)$ if the definitions of $a$ and $\ell$ are generalized as stated in the claim.

If $i-j=d$, then $\Tor_i^{R_n}(R_n/I_n,\kk)_j\ne 0$ only when $i=n-1$ and when $i=n$. Using \cref{thm: main}(2), we have:
\begin{align*}
\beta_{i,i+d}(\I,s)&=\dim_\kk \Tor_i^{R_{i}}(R_i/I_{i},\kk)_{i+d} s^{i}+\dim_\kk \Tor_i^{R_{i+1}}(R_{i+1}/I_{i+1},\kk)_{i+d} s^{i+1}\\
&=as^i+\ell s^{i+1}\,.
\end{align*}

If $i-j=d-1$, then $\Tor_i^{R_n}(R_n/I_n,\kk)_j\ne 0$ only when $i\le n$. Using the formulas in \cref{thm: main}(2) and the notation $b_i=\binom{i+d-2}{d-1}-ai+\ell$, we obtain for all $i\ge 1$: 
\begin{align*}
&\beta_{i,i+d-1}(\I,s)
=\sum_{n>i}\dim_\kk\Tor_i^{R_n}(R_n/I_n,\kk)_{i+d-1}s^n+\dim_\kk\Tor_i^{R_i}(R_i/I_i,\kk)_{i+d-1}s^i\\
&=\sum_{n> i}\left( 
\binom{n+d-1}{d+i-1}\binom{d+i-2}{i-1}-a\binom{n}{i-1} \right)s^n+b_is^i\\
&=\binom{d+i-2}{i-1}\left(\sum_{n>i}\binom{n+d-1}{d+i-1}s^n\right)-a\sum_{n>i}\binom{n}{i-1}s^n+b_is^i\\
&=\binom{d+i-2}{i-1}\cdot s^i\left(\sum_{m\ge 1}\binom{m+i+d-1}{d+i-1}s^{m}\right)-as^{i-1}\left(\sum_{m\ge 2}\binom{m+i-1}{i-1}s^m\right)+b_is^i\\
&=\binom{d+i-2}{i-1}\!\left(\frac{s^i}{(1\!-\!s)^{i+d}}-s^i\right)\!-\!a\!\left(\frac{s^{i-1}}{(1\!-\!s)^{i}}-\!s^{i-1}\!-is^i\!\right)+\!\binom{i+d-2}{d-1}s^i-ais^i+\ell s^i\\
   &=\binom{d+i-2}{i-1}\frac{s^i}{(1-s)^{i+d}}-\frac{a}{s}\left(\frac{s^{i}}{(1-s)^{i}}-s^{i}\right)+\ell s^i\,.  
\end{align*}

Also, observe that 
\[
\beta_{0,0}(\I,s)=\sum_{n\ge 0}s^n=\frac{1}{1-s}\,.
\]
We have thus 
\begin{align*}
&P_I(s,t,u)=\frac{1}{1-s}+\sum_{i\ge 1}\beta_{i,{i+d-1}}(\I,s)t^iu^{i+d-1}+\sum_{i\ge 1}\beta_{i,i+d}(\I,s)t^iu^{i+d}\\
&=\frac{1}{1-s}\!+u^{d-1}\!\left(\sum_{i\ge 1}\binom{d+i-2}{d-1}\frac{(stu)^i}{(1-s)^{i+d}}-\frac{a}{s}\sum_{i\ge 1}\left(\frac{(stu)^i}{(1-s)^{i}}-\!(stu)^i\right)+\ell\sum_{i\ge 1}(stu)^i\!\right)\\
&\qquad+u^d\left(\sum_{i\ge 1} a(stu)^i+\ell s (stu)^i\right).
\end{align*}
We compute some of the terms in this sum separately: 
\begin{align*}
\sum_{i\ge 1}\binom{d+i-2}{d-1}\frac{(stu)^i}{(1-s)^{i+d}}&=\frac{stu}{(1-s)^{d+1}}\sum_{m\ge 0}\binom{d-1+m}{d-1}\left(\frac{stu}{1-s}\right)^m\\
&=\frac{stu}{(1-s)^{d+1}}\frac{1}{\left(1-\frac{stu}{1-s}\right)^d}=\frac{stu}{(1-s)(1-s-stu)^d}  
\end{align*}
\begin{align*}
\sum_{i\ge 1}\left(\frac{(stu)^i}{(1-s)^{i}}-(stu)^i\right)&=\left(\frac{1}{1-\frac{stu}{1-s}}-\frac{1}{1-stu}\right)=\frac{1-s}{1-s-stu}-\frac{1}{1-stu}\\
&=\frac{s^2tu}{(1-s-stu)(1-stu)}\,.
\end{align*}

In conclusion, the bigraded equivariant Poincar\'e series is
\begin{equation*}
P_\I(s,t,u)\!=\!\frac{1}{1-s}\!+\!stu^d\left(\frac{1}{(1\!-\!s)(1\!-\!s\!-\!stu)^d}\!-\!\frac{a}{(1\!-\!s\!-\!stu)(1\!-\!stu)}\!+\!\frac{\ell+ au +\ell su}{1-stu}\right).
\end{equation*}
\end{proof}

\paragraph{ \bf Acknowledgement} This material is based upon work supported by the National Science Foundation under Grant No. DMS-1928930 and by the Alfred P. Sloan Foundation under grant G-2021-16778, while the authors were in residence at the Simons Laufer Mathematical Sciences Institute (formerly MSRI) in Berkeley, California, during the Spring 2024 semester and separately by the Fields Institute  the authors visited during the Spring 2025 thematic program in Commutative Algebra and its Applications. The first author is partially supported by NSF DMS--2401482.

\bigskip
\bibliographystyle{abbrv}
\bibliography{references}

\begin{thebibliography}{10}

\bibitem{AbdallahSchenck}
N.~Abdallah and H.~Schenck.
\newblock Free resolutions and {L}efschetz properties of some {A}rtin
  {G}orenstein rings of codimension four.
\newblock {\em J. Symbolic Comput.}, 121:Paper No. 102257, 13, 2024.

\bibitem{Anick}
D.~J. Anick.
\newblock Thin algebras of embedding dimension three.
\newblock {\em J. Algebra}, 100(1):235--259, 1986.

\bibitem{Hillar}
M.~Aschenbrenner and C.~J. Hillar.
\newblock Finite generation of symmetric ideals.
\newblock {\em Trans. Amer. Math. Soc.}, 359(11):5171--5192, 2007.

\bibitem{Biermann}
J.~Biermann, H.~de~Alba, F.~Galetto, S.~Murai, U.~Nagel, A.~O'Keefe,
  T.~R\"omer, and A.~Seceleanu.
\newblock Betti numbers of symmetric shifted ideals.
\newblock {\em J. Algebra}, 560:312--342, 2020.

\bibitem{Boij}
M.~Boij.
\newblock Betti numbers of compressed level algebras.
\newblock {\em J. Pure Appl. Algebra}, 134(2):111--131, 1999.

\bibitem{Cohen}
D.~E. Cohen.
\newblock On the laws of a metabelian variety.
\newblock {\em J. Algebra}, 5:267--273, 1967.

\bibitem{Draisma}
J.~Draisma, M.~Laso\'n, and A.~Leykin.
\newblock Stillman's conjecture via generic initial ideals.
\newblock {\em Comm. Algebra}, 47(6):2384--2395, 2019.

\bibitem{Eisenbud}
D.~Eisenbud.
\newblock {\em Commutative algebra}, volume 150 of {\em Graduate Texts in
  Mathematics}.
\newblock Springer-Verlag, New York, 1995.
\newblock With a view toward algebraic geometry.

\bibitem{Erman}
D.~Erman, S.~V. Sam, and A.~Snowden.
\newblock Big polynomial rings and {S}tillman's conjecture.
\newblock {\em Invent. Math.}, 218(2):413--439, 2019.

\bibitem{Froberg}
R.~Fr\"oberg.
\newblock An inequality for {H}ilbert series of graded algebras.
\newblock {\em Math. Scand.}, 56(2):117--144, 1985.

\bibitem{FrobergLaksov}
R.~Fr\"oberg and D.~Laksov.
\newblock Compressed algebras.
\newblock In {\em Complete intersections ({A}cireale, 1983)}, volume 1092 of
  {\em Lecture Notes in Math.}, pages 121--151. Springer, Berlin, 1984.

\bibitem{FrobergLofwall}
R.~Fr\"oberg and C.~L\"ofwall.
\newblock On {H}ilbert series for commutative and noncommutative graded
  algebras.
\newblock {\em J. Pure Appl. Algebra}, 76(1):33--38, 1991.

\bibitem{Galetto}
F.~Galetto.
\newblock On the ideal generated by all squarefree monomials of a given degree.
\newblock {\em J. Commut. Algebra}, 12(2):199--215, 2020.

\bibitem{Green}
E.~L. Green.
\newblock Complete intersections and {G}orenstein ideals.
\newblock {\em J. Algebra}, 52(1):264--273, 1978.

\bibitem{HSS24}
M.~Harada, A.~Seceleanu, and L.~M. \c~Sega.
\newblock The minimal free resolution of a general principal symmetric ideal.
\newblock {\em Trans. Amer. Math. Soc.}, 378(3):1831--1882, 2025.

\bibitem{HS}
C.~J. Hillar and S.~Sullivant.
\newblock Finite {G}r\"obner bases in infinite dimensional polynomial rings and
  applications.
\newblock {\em Adv. Math.}, 229(1):1--25, 2012.

\bibitem{Iarrobino}
A.~Iarrobino.
\newblock Compressed algebras: {A}rtin algebras having given socle degrees and
  maximal length.
\newblock {\em Trans. Amer. Math. Soc.}, 285(1):337--378, 1984.

\bibitem{EI}
A.~Iarrobino and J.~Emsalem.
\newblock Some zero-dimensional generic singularities; finite algebras having
  small tangent space.
\newblock {\em Compositio Math.}, 36(2):145--188, 1978.

\bibitem{IK}
A.~Iarrobino and V.~Kanev.
\newblock {\em Power sums, {G}orenstein algebras, and determinantal loci},
  volume 1721 of {\em Lecture Notes in Mathematics}.
\newblock Springer-Verlag, Berlin, 1999.
\newblock Appendix C by Iarrobino and Steven L. Kleiman.

\bibitem{KLR}
M.~Juhnke-Kubitzke, D.~V. Le, and T.~R\"omer.
\newblock Asymptotic behavior of symmetric ideals: a brief survey.
\newblock In {\em Combinatorial structures in algebra and geometry}, volume 331
  of {\em Springer Proc. Math. Stat.}, pages 73--94. Springer, Cham, [2020]
  \copyright 2020.

\bibitem{Kretschmer}
A.~Kretschmer.
\newblock When are symmetric ideals monomial?
\newblock {\em J. Commut. Algebra}, 15(3):367--376, 2023.

\bibitem{Nagel1}
D.~V. Le, U.~Nagel, H.~D. Nguyen, and T.~R\"omer.
\newblock Codimension and projective dimension up to symmetry.
\newblock {\em Math. Nachr.}, 293(2):346--362, 2020.

\bibitem{Nagel2}
D.~V. Le, U.~Nagel, H.~D. Nguyen, and T.~R\"omer.
\newblock Castelnuovo-{M}umford regularity up to symmetry.
\newblock {\em Int. Math. Res. Not. IMRN}, (14):11010--11049, 2021.

\bibitem{MM}
J.~Migliore and R.~M. Mir\'o-Roig.
\newblock On the minimal free resolution of {$n+1$} general forms.
\newblock {\em Trans. Amer. Math. Soc.}, 355(1):1--36, 2003.

\bibitem{Satoshi1}
S.~Murai.
\newblock Betti tables of monomial ideals fixed by permutations of the
  variables.
\newblock {\em Trans. Amer. Math. Soc.}, 373(10):7087--7107, 2020.

\bibitem{Satoshi-Raicu}
S.~Murai and C.~Raicu.
\newblock An equivariant {H}ochster's formula for {$\sym_n$}-invariant monomial
  ideals.
\newblock {\em J. Lond. Math. Soc. (2)}, 105(3):1974--2010, 2022.

\bibitem{Nagel3}
U.~Nagel.
\newblock Rationality of equivariant {H}ilbert series and asymptotic
  properties.
\newblock {\em Trans. Amer. Math. Soc.}, 374(10):7313--7357, 2021.

\bibitem{Nagel}
U.~Nagel.
\newblock Rationality of equivariant {H}ilbert series and asymptotic
  properties.
\newblock {\em Trans. Amer. Math. Soc.}, 374(10):7313--7357, 2021.

\bibitem{NagelRomer}
U.~Nagel and T.~R\"omer.
\newblock Equivariant {H}ilbert series in non-noetherian polynomial rings.
\newblock {\em J. Algebra}, 486:204--245, 2017.

\bibitem{NagelRomerFI}
U.~Nagel and T.~R\"omer.
\newblock F{I}- and {OI}-modules with varying coefficients.
\newblock {\em J. Algebra}, 535:286--322, 2019.

\bibitem{Nenashev}
G.~Nenashev.
\newblock A note on {F}r\"oberg's conjecture for forms of equal degrees.
\newblock {\em C. R. Math. Acad. Sci. Paris}, 355(3):272--276, 2017.

\bibitem{Raicu1}
C.~Raicu.
\newblock Regularity of {$\sym_n$}-invariant monomial ideals.
\newblock {\em J. Combin. Theory Ser. A}, 177:Paper No. 105307, 34, 2021.

\bibitem{RS}
M.~E. Rossi and L.~M. \c~Sega.
\newblock Poincar\'e{} series of modules over compressed {G}orenstein local
  rings.
\newblock {\em Adv. Math.}, 259:421--447, 2014.

\bibitem{Stanley}
R.~P. Stanley.
\newblock Hilbert functions of graded algebras.
\newblock {\em Advances in Math.}, 28(1):57--83, 1978.

\bibitem{Walker}
N.~Walker.
\newblock On general principal symmetric ideals.
\newblock {\em \href{https://arxiv.org/abs/2505.21802}{ArXiv:2505.21802}},
  2025.

\bibitem{Wiebe}
A.~Wiebe.
\newblock The {L}efschetz property for componentwise linear ideals and
  {G}otzmann ideals.
\newblock {\em Comm. Algebra}, 32(12):4601--4611, 2004.

\end{thebibliography}
\end{document}